\documentclass[12pt,reqno]{amsart}

\usepackage{amsthm,amsmath,amssymb,mathrsfs}
\usepackage[normalem]{ulem}
\usepackage{color}
\newtheorem{teor}{Theorem}[section]
\newtheorem{thm}[teor]{Theorem}
\newtheorem{lemm}[teor]{Lemma}
\newtheorem{lem}[teor]{Lemma}
\newtheorem{prop}[teor]{Proposition}
\newtheorem{coro}[teor]{Corollary}

\theoremstyle{definition}
\newtheorem{defin}[teor]{Definition}
\newtheorem{df}[teor]{Definition}
\newtheorem{defi}[teor]{Definition}

\theoremstyle{remark}
\newtheorem{osse}[teor]{Remark}
\newtheorem{rem}[teor]{Remark}

\newtheorem{prob}[teor]{Problem}
\newtheorem{assu}[teor]{Assumption}

\newcommand{\bele}{\begin{lemm}\begin{sl}}
\newcommand{\enle}{\end{sl}\end{lemm}}
\newcommand{\bedef}{\begin{defi}\begin{sl}}
\newcommand{\eddef}{\end{sl}\end{defi}}
\newcommand{\bete}{\begin{teor}\begin{sl}}
\newcommand{\ente}{\end{sl}\end{teor}}
\newcommand{\beos}{\begin{osse}\begin{rm}}
\newcommand{\eddos}{\end{rm}\end{osse}}
\newcommand{\beas}{\begin{assu}\begin{rm}}
\newcommand{\eddas}{\end{rm}\end{assu}}
\newcommand{\bepr}{\begin{prop}\begin{sl}}
\newcommand{\empr}{\end{sl}\end{prop}}
\newcommand{\bepro}{\begin{prob}\begin{rm}}
\newcommand{\empro}{\end{rm}\end{prob}}
\newcommand{\bede}{\begin{defin}\begin{sl}}
\newcommand{\edde}{\end{sl}\end{defin}}
\newcommand{\beco}{\begin{coro}\begin{sl}}
\newcommand{\enco}{\end{sl}\end{coro}}

\newcommand{\quext}{\quad\text}

\newcommand{\de}{\partial}

\newcommand{\RR}{\mathbb{R}}

\newcommand{\NN}{\mathbb{N}}

\newcommand{\phj}{\varphi}
\newcommand{\modu}{\varrho_{\phj,\Omega}}
\newcommand{\modud}{\varrho_{\phj^*,\Omega}}
\newcommand{\Modu}{\varrho_{\phj,Q}}
\newcommand{\Modud}{\varrho_{\phj^*,Q}}

\newcommand{\Vp}{V^*}
\newcommand{\calVp}{{\mathcal V}^*}
\newcommand{\fQ}{\mathcal{Q}}

\newcommand{\beeq}[1]{\begin{equation}\label{#1}}
\newcommand{\eddeq}{\end{equation}}

\newcommand{\beeqa}[1]{\begin{eqnarray}\label{#1}}
\newcommand{\eddeqa}{\end{eqnarray}}

\newcommand{\beal}[1]{\begin{align}\label{#1}}
\newcommand{\eddal}{\end{align}}

\newcommand{\bespl}[1]{\begin{split}\label{#1}}
\newcommand{\edspl}{\end{split}}

\newcommand{\bega}[1]{\begin{gather}\label{#1}}
\newcommand{\edga}{\end{gather}}

\newcommand{\beeqax}{\begin{eqnarray*}}
\newcommand{\eddeqax}{\end{eqnarray*}}

\newcommand{\no}{\nonumber}

\newcommand{\beeqao}{\begin{eqnarray}\no}
\newcommand{\bealo}{\begin{align}\no}
\newcommand{\besplo}{\begin{split}\no}
\newcommand{\begao}{\begin{gather}\no}

\newcommand{\duav}[1]{\langle{#1}\rangle}

\newcommand{\duaV}[1]{\langle\!\langle{#1}\rangle\!\rangle}

\newcommand{\+}{\hspace{1pt}}

\newcommand{\io}{\int_\Omega}
\newcommand{\iE}{\int_E}

\newcommand{\iTT}{\int_0^T}

\newcommand{\iTo}{\iint_Q}

\newcommand{\epsi}{\varepsilon}

\newcommand{\lla}{_{\lambda}}

		     \def\R{\mathbb R}
		     \def\N{\mathbb N}

\newcommand{\fhi}{\mathscr{E}}
\newcommand{\Fhi}{\mathbb{E}}

\newcommand{\lhs}{left-hand side}
\newcommand{\rhs}{right-hand side}

\DeclareMathOperator{\deriv}{d}

\DeclareMathOperator{\exte}{ext}

\newcommand{\Bext}{B_{\+\exte}}

\newcommand{\calA}{{\mathcal A}}

\newcommand{\calJ}{{\mathcal J}}

\newcommand{\calM}{{\mathcal M}}

\newcommand{\calB}{{\mathcal B}}

\newcommand{\calS}{{\mathcal S}}

\newcommand{\calV}{{\mathcal V}}

\newcommand{\dit}{\deriv\!t}
\newcommand{\dis}{\deriv\!s}
\newcommand{\dix}{\deriv\!x}

\newcommand{\Lpsot}{L^\phj(\Omega)}
\newcommand{\Lpso}{L^\phj(\Omega)}

\newcommand{\Lpsq}{L^\phj(Q)}

\newcommand{\Lpsostar}{L^{\phj^*}(\Omega)}

\newcommand{\deo}{\partial_\Omega}
\newcommand{\deq}{\partial_Q}

\numberwithin{equation}{section}

\renewcommand{\d}{\mathrm{d}}
\newcommand{\e}{\mathrm{e}}
\newcommand{\lam}{\lambda}
\newcommand{\vep}{\varepsilon}
\DeclareMathOperator*{\esssup}{ess\,sup}
\DeclareMathOperator*{\essinf}{ess\,inf}

\allowdisplaybreaks[4]

\title[Doubly-nonlinear evolution in Musielak-Orlicz spaces]{On a class of doubly-nonlinear evolution equations\\in Musielak-Orlicz spaces}

\author{Goro Akagi}
\address[Goro Akagi]{Mathematical Institute and Graduate School of Sciences, Tohoku University, Aoba, Sendai 980-8578 Japan}
\email{goro.akagi@tohoku.ac.jp}

\author{Giulio Schimperna}
\address[Giulio Schimperna]{Dipartimento di Matematica, Universit\`a di Pavia,
and Istituto di Matematica Applicata e Tecnologie Informatiche ``Enrico Magenes'' (IMATI),
Via Ferrata~5, I-27100 Pavia, Italy}
\email{giusch04@unipv.it}

\date{\today}


\usepackage{version}	
\begin{document}


\begin{abstract}
 This paper is concerned with a parabolic evolution equation of the form $A(u_t) + B(u) = f$, settled in a smooth bounded domain of $\RR^d$, $d\ge 1$, and complemented with the initial conditions and with (for simplicity) homogeneous Dirichlet boundary conditions. Here, $-B$ stands for a diffusion operator, possibly nonlinear, which may range in a very wide class, including the Laplacian, the $m$-Laplacian for suitable $m\in (1,\infty)$, the ``variable-exponent'' $m(x)$-Laplacian, or even some fractional order operators. The operator $A$ is assumed to be in the form $[A(v)](x,t)=\alpha(x,v(x,t))$ with $\alpha$ being measurable in $x$ and maximal monotone in $v$. The main results are devoted to proving existence of weak solutions for a wide class of functions $\alpha$ that extends the setting considered in previous results related to the variable exponent case where $\alpha(x,v)=|v(x)|^{p(x)-2}v(x)$. To this end, a theory of subdifferential operators will be established in Musielak-Orlicz spaces satisfying structure conditions of the so-called $\Delta_2$-type and a framework for approximating maximal monotone operators acting in that class of spaces will also be developed. Such a theory is then applied to provide an existence result for a specific equation, but it may have an independent interest in itself. Finally, the existence result is illustrated by presenting a number of specific equations (and, correspondingly, of operators $A$, $B$) to which the result can be applied.
\end{abstract}

\maketitle

\section{Introduction}
\label{sec:intro}

\emph{Doubly-nonlinear evolution equations} have been studied for more than half a century. In principle, they may be classified into two forms 
(see~\cite{Visintin}, cf.\ also~\cite{Roubicek}): the first originates from a generalization of nonlinear problems
such as the \emph{fast diffusion} and \emph{porous medium equations}. In that case there appear two nonlinearities acting on the unknown itself as follows:
\begin{equation}
\label{dne0}
\partial_t A(u) + B(u) = 0,
\end{equation}
where $\partial_t$ denotes the time-derivative and $A$ and $B$ are (possibly) nonlinear operators acting on a proper function space 
and whose typical example is given as $A(u)=|u|^{p-2}u$, a power of the unknown $u=u(x,t)$ with $1 < p < +\infty$, 
and $B(u)=-\Delta_m u := - \mathrm{div} \left( |\nabla u|^{m-2}\nabla u\right)$ with the so-called \emph{$m$-Laplace}
operator $\Delta_m$. On the other hand, the second type of doubly nonlinear equation has been introduced in
a celebrated work by V.~Barbu~\cite{B75} dealing with the following evolutionary problem:
\begin{equation}
\label{dne} A(u_t) + B(u) = 0,
\end{equation}
where $u_t = \partial_t u$ (see also~\cite{Arai,CV,Colli,GO3,SSS,Roubicek,G15} and references therein).
Concrete examples can also be provided by taking the same choices of $A$ and $B$ above, that is,
\begin{equation*}
 |u_t|^{p-2}u_t - \Delta_m u = 0.
\end{equation*}
So far, equation~\eqref{dne} seems to have been studied less extensively compared to~\eqref{dne0}, even though a number of recent works~\cite{Hy17,Hy19,HyLi16,HyLi17,HyLi19} have been devoted to equations of the form~\eqref{dne}
occurring in the modelization of the so-called \emph{strongly irreversible} 
(or \emph{unidirectional}) \emph{processes} and of \emph{rate-independent processes} (see, e.g.,~\cite{AK19,KRZ13,MiRo15}). 

The present paper  is also devoted to 
studying  the latter doubly-nonlinear problem \eqref{dne}.
 In order to explain the novelties of our results, we need to present some overview
of the previous works dealing with~\eqref{dne}. 
In~\cite{B75} and subsequent studies (see~\cite{Arai}), existence of a strong solution is 
proved for the following abstract Cauchy problem in a Hilbert space $H$,
$$
\partial \psi(u_t(t)) + \partial \varphi(u(t)) \ni f(t) \ \mbox{ in } H , \quad 0 < t < T, \quad u(0)=u_0,
$$
where $\partial \psi$ and $\partial \varphi$ denote the \emph{subdifferentials} of $\psi$ and $\varphi$, 
respectively, and $f : (0,T) \to H$ and $u_0 \in H$ are given data, assuming some additional monotonicity 
(namely, the so-called \emph{$\partial\psi$-monotonicity}) of $\partial\varphi$ as well as  some differentiability 
of $f$.  These structure assumptions are essential to ensure the validity of the results in~\cite{B75,Arai}. 
On the other hand, Colli and Visintin~\cite{CV} shed new light on this field; they presented a different 
framework, which can cover the equation
\begin{equation}\label{cv}
A(u_t(t)) + \partial \varphi(u(t)) \ni f(t) \ \mbox{ in } H, \quad 0 < t < T, \quad u(0)=u_0,
\end{equation}
for any maximal monotone operator $A : H \to H$ satisfying an affine growth condition (see \eqref{hp:cv} below with $p=2$ and $V=H$) and for any $f \in L^2(0,T;H)$.  Afterwards, these results were extended in~\cite{Colli} to any maximal monotone operator $A : V \to V^*$ satisfying a $p$-growth condition for any $p \in (1,+\infty)$ in a reflexive Banach space $V$ and for $f \in L^{p'}(0,T;V^*)$ with the H\"older conjugate $p'$ of $p$ and the dual space $V^*$ of $V$. Moreover, 
 we observe that the theory developed in~\cite{CV,Colli}, being more adaptable to deal with perturbations, can also be applied to the so-called phase-field models, which are characterized by the presence of a further semilinear term of the form $h(u)$, where the function $h$ is generally nonlinear and may have a singular character. On the other hand, the growth condition on $A$ plays an essential role in the development of the theory. To be precise, in~\cite{Colli}, the operator $A : V \to V^*$ is supposed to satisfy
\begin{equation}\label{hp:cv}
\alpha \|u\|_V^p \leq \langle A(u), u \rangle + C_1, \quad \|A(u)\|_{V^*}^{p'} \leq C_2 (\|u\|_V^p + 1) \quad \mbox{ for } \ u \in V,
\end{equation}
where $\alpha > 0$ and $C_1,C_2 \geq 0$ are constants, and the relation among exponents in these conditions
is indispensable in place of the differentiability of $f$ and the $\partial \psi$-monotonicity of $\partial \varphi$ 
as in~\cite{B75,Arai}. Then one may wonder whether or not the $p$-growth condition might be relaxed or generalized, 
but without assuming the differentiability of $f$ as in~\cite{B75,Arai}. 

In order to discuss such a question, we first recall that 
the following doubly-nonlinear parabolic equation is studied in~\cite{AS} as a toy model:
\begin{alignat}{4}
|u_t|^{p(x)-2}u_t(x,t) - \Delta_{m(x)} u(x,t) &= f(x,t) \quad &&\mbox{ for } \ x\in \Omega, \ t > 0,\label{pde0}\\
u(x,t) &= 0 \quad &&\mbox{ for } \ x \in \partial \Omega, \ t > 0,\label{bc0}\\
u(x,0) &= u_0(x) \quad &&\mbox{ for } \ x \in \Omega,\label{ic0}
\end{alignat}
where $p = p(x), \, m = m(x) : \Omega \to [1,+\infty]$ are \emph{variable exponents} (i.e., measurable functions in $\Omega$) satisfying
\begin{equation*}
1 < p^- := \essinf_{x \in \Omega} p(x) \leq p^+ := \esssup_{x \in \Omega} p(x) < +\infty, \quad 1 < m^- \leq m^+ < +\infty
\end{equation*}
and $\Delta_{m(x)}$ stands for the so-called \emph{$m(x)$-Laplace operator} given by
$$
\Delta_{m(x)} w(x) := \mathrm{div} \, \left(|\nabla w(x)|^{m(x)-2} \nabla w(x)\right).
$$
Indeed, although the power nonlinearity $r \in \R \mapsto |r|^{p(x)-2}r$ is homogeneous of degree $p(x)-1$ at each $x \in \Omega$, the Cauchy-Dirichlet problem \eqref{pde0}--\eqref{ic0} is reduced into the abstract Cauchy problem for \eqref{cv} posed on $V = L^{p(x)}(\Omega)$, which is the so-called \emph{variable exponent Lebesgue space} (see~\cite{DHHR}), and the operator $A$ (defined as $A(w) = |w|^{p(x)-2}w$ for $v \in V$) is no longer consistent with the assumption \eqref{hp:cv} (moreover, it is also beyond the scope of~\cite{B75, Arai} due to the $x$-dependence of the nonlinearity, since it violates the $\partial \psi$-monotonicity of $\partial \varphi$). In~\cite{AS}, the theory of maximal monotone operators (in particular, subdifferential calculus) is customized for the variable exponent Lebesgue space 
setting, and,  based on this generalization,  existence of strong solutions to \eqref{pde0}--\eqref{ic0} is proved with maximal regularity (i.e., for $f$ lying on a certain class $\calV^*$, both $|u_t|^{p(x)-2}u_t$ and $\Delta_{m(x)}u$ belong to the same $\calV^*$).  A relevant mathematical difficulty resides in the  mismatch of the equation with the standard frame for evolution equations based on Lebesgue-Bochner spaces, say $L^p(0,T;V)$, where functions of $x$ and $t$ are regarded as vector-valued functions of $t$ only and which gives a better fit to energy methods such as chain-rule formula. Indeed, in order to reformulate the PDE \eqref{pde0} in a Bochner-Lebesgue space setting, one cannot fully figure out the integrability of the nonlinear term due to the inhomogeneity of variable exponents (for instance, $L^{p(x)}(0,T;L^{p(x)}(\Omega))$ makes no longer sense with variable exponents), and such a defect fatally violates the $p$-growth condition used in~\cite{CV,Colli}. On the other hand, one can fully extract the integrability in the frame of space-time (variable exponent) Lebesgue spaces, say $L^{p(x)}(\Omega\times(0,T))$,  where, however, energy methods may no longer be applied in a standard way. To overcome such a difficulty, in~\cite{AS} we introduced a customized theory of subdifferential calculus, in a somehow mixed framework where fine properties of variable exponent Lebesgue spaces (see~\cite{DHHR}) play a crucial role in order to preserve the availability of variational methods. 

In the present paper we shall further generalize the results of \cite{AS} by keeping, as in that paper, the dependence on $x$ of the nonlinear term acting on $u_t$, but at the same time weakening its structure properties. Indeed, power functions with variable exponents still enjoy homogeneity at each $x \in \Omega$, whereas one may consider a case where, for some, or all, $x\in\Omega$, $A$ may act on its target $v$, for example, as $v \log(1+|v|)$, i.e., the power-growth condition in the sense of \eqref{hp:cv} does no longer hold. To present our framework in a more rigorous way,  we let $\Omega$ be a smooth bounded domain in $\RR^N$ with $N\ge 1$, let $T>0$ be an assigned final time and consider the following doubly-nonlinear evolution equation for the unknown $u:(0,T)\times \Omega\to \RR$\/:
\begin{equation}\label{proto-eqn}
  \alpha(x,u_t(x,t)) + [B(u)](x,t) = f(x,t)  \quad \text{ for } \ x \in \Omega, \ t \in (0,T),
\end{equation}
where the function $f=f(x,t)$ represents an assigned forcing term. Here $\alpha=\alpha(x,r) : \Omega \times \R \to \R$ is a nonlinear function which is measurable in $x$ and maximal monotone in $r$ and $B$ is an operator of subdifferential type in a proper function space (associated with the function $\alpha(x,r)$ as we shall see later); as a concrete example, one may consider 
$$
B(u) = -\Delta_{m(x)}u
$$
equipped with the homogeneous Dirichlet condition. In addition, the precise assumptions on $\alpha$ and $B$ will be thoroughly discussed below and rigorously detailed in the next section; typical examples of $\alpha(x,r)$ which we have in mind are
$$
\alpha(x,r) = |r|^{p(x)-2}r, \quad |r|^{p(x)-2} r [\log (|r|+1)]^{q(x)}, \quad \alpha(x) |r|^{p(x)-2}r + \beta(x) |r|^{q(x)-2}r
$$
for variable exponents $p(x)$, $q(x)$ and nonnegative $\alpha, \beta \in L^\infty(\Omega)$ satisfying $\alpha+\beta > 0$ in $\Omega$. These examples, except for the first one, are beyond the scope of the result in~\cite{AS} as well as those in~\cite{B75, Arai, CV, Colli}. Equation \eqref{proto-eqn} is complemented with the initial condition $u|_{t=0}=u_0$, whereas the boundary conditions will be incorporated in the definition of $B$ and in the chosen functional setting. The main purpose of the present paper is proving existence of strong solutions to \eqref{proto-eqn} under  assumptions on $\alpha$ and $B$ that are weaker and more general compared to previous results. 
To this end, we shall employ the so-called \emph{Musielak-Orlicz spaces}, which are a general category of function spaces including variable exponent Lebesgue spaces and having Orlicz spaces as a subclass. 

Constructing a strong solution for \eqref{proto-eqn} is also relevant to a theory of metric gradient flows (see~\cite{AGS}), where a \emph{curve of maximal slope} $u : [0,T] \to (X,\d)$ of an energy $\varphi : X \to (-\infty,+\infty]$ in a metric space $(X,\d)$ is constructed based on the \emph{minimizing movement scheme}; more precisely, $u : [0,T] \to X$ is a limit of the piecewise constant interpolant $\hat u_\tau : [0,T] \to X$ of discretized solutions $\{u_j\}$ which minimize the functionals,
$$
w \in X \mapsto \frac 1{2\tau} \d(w,u_{j-1})^2 + \varphi(w), \quad j = 1,2,\ldots, K
$$
with the time-step $\tau = T/K > 0$ and an initial datum $u_0 \in X$. Moreover, in~\cite{AGS}, for $1<p<+\infty$, the notion of \emph{$p$-curve of maximal slope} is introduced as a natural generalization and it can be constructed by performing the minimizing movement scheme with the functional above replaced by 
$$
w \in X \mapsto \frac 1{p\tau} \d(w,u_{j-1})^p + \varphi(w).
$$
In particular, if $X = V$ is a reflexive Banach space (i.e., $\d(u,v) = \|u-v\|_V$), one can derive a gradient system (with a duality mapping $F_V : V \to V^*$), 
$$
\|u_t\|_V^{p-2} F_V(u_t(t)) + \partial \varphi(u(t)) \ni 0 \ \mbox{ in } V^*, \quad 0 < t < T,
$$
which generates a $p$-curve of maximal slope as a solution $u :[0,T] \to V$ and which is reduced to the doubly-nonlinear evolution equation \eqref{cv} in $V^*$ satisfying \eqref{hp:cv}.  Indeed, in order to address \eqref{proto-eqn}, 
we shall take a potential function $\phj(x,r)$ of $\alpha(x,r)$, i.e., $\alpha(x,r) = \partial \phj(x,r)$, and construct a set of discretized solutions $\{u_j\}$ by minimizing the functionals
$$
w \in V \mapsto \tau \int_\Omega \phj\left(x,\frac{w(x)-u_{j-1}(x)}\tau \right) \, \d x + \fhi(w), \quad j = 1,2,\ldots,K,
$$
where $\fhi : V \to [0,+\infty]$ is a potential functional of the operator $B$ (i.e., $B = \partial \fhi$), defined on a Musielak-Orlicz space $V$  which is defined by exploiting the specific expression of $\phj$ as a 
\emph{generalized $\Phi$-function} (we refer to \cite{DHHR} for the underlying function space theory). 
Here we remark that the first term of the functional is a \emph{modular} of the Musielak-Orlicz space $V$ and is no longer homogeneous unlike the case of $p$-curves of maximal slope in a Banach space. Hence the construction of solutions to \eqref{proto-eqn} is more like that of a $p$-curve of maximal slope in a metric space, where the homogeneity (or even asymptotically affine growth) is no longer available for the metric; on the other hand,  we have to remark that,  for \eqref{proto-eqn}, the underlying (Musielak-Orlicz) space $V$ still has a linear structure.

To achieve the goal of the present paper, we shall introduce two base spaces: one, $V=L^\phj(\Omega)$, is a Musielak-Orlicz space in  the space variables only, while the other one, $\calV=L^\phj(\Omega\times(0,T))$, keeps that structure with respect to both space and time variables. As in the variable exponent case, however, there is a regularity gap between those vector-valued spaces on $(0,T)$ that can actually be constructed by keeping $V$ as a target space and the space $\calV$. Indeed, the latter is the natural space for weak solutions, but, on the other hand, it \emph{cannot} be viewed as a Lebesgue-Bochner space. In order to bridge this regularity mismatch, we shall develop a subdifferential calculus in a mixed framework and, based on this machinery, we shall prove the main results of this paper, i.e., existence of strong solutions to the Cauchy problem for \eqref{proto-eqn} with maximal regularity under certain assumptions.  It is worth stressing that the subdifferential calculus we shall introduce has an independent interest in itself, and its applications are not at all restricted to dealing with the specific equation \eqref{proto-eqn}.  
We also refer the reader to~\cite{Zhang18} for a different approach to extend Musielak-Orlicz spaces in space variables only to certain space-time spaces, which look more like a generalization of Bochner spaces. 

The paper is organized as follows. In the next section, we shall recall the minimum preliminary material on the theory of modular and Musielak-Orlicz spaces for later use. In Section \ref{sec:main}, we shall present precise assumptions on the operators $\alpha$ and $B$ and state our main result (see Theorem \ref{teo:main} below) regarding existence of strong solutions to the Cauchy problem for \eqref{proto-eqn}. Sections \ref{sec:tools}--\ref{subsec:chain} are devoted to setting up a machinery related with subdifferential calculus in Musielak-Orlicz spaces and to proving a number of auxiliary lemmas. More precisely, Section \ref{sec:tools} is concerned with generic lemmas which hold for Musielak-Orlicz spaces satisfying the $\Delta_2$-condition; in Section \ref{subsec:convex}, the  basic tools of subdifferential calculus will be set up; Section \ref{subsec:embe} provides a suitable extension of the Aubin-Lions lemma; finally, in Section \ref{subsec:chain}, we shall provide a chain-rule for subdifferentials in the Musielak-Orlicz space setting by customizing related notions such as resolvent, Yosida approximation and Moreau-Yosida regularization for convex functionals. The customized chain-rule formula will play a crucial role (exactly in~\eqref{form:17} below) in  the proof of the main result, which will be outlined  in Section \ref{sec:proof}. Moreover, the main result will be generalized in Section \ref{sec:gen} (see Theorem \ref{T:gen} below). The final section exhibits concrete examples of doubly-nonlinear PDEs which fall within the scope of the theory developed 
in the previous sections.

\smallskip
\noindent
{\bf Notation.} For each $t > 0$, we shall often write $u(t)$, which is a function in space and may be an element of a function space, instead of $u(\cdot,t)$ for functions $u=u(x,t)$ in space and time with values in $\R$. Moreover, we shall denote by $C$ a non-negative constant which does not depend on the elements of the corresponding space or set and may vary from line to line. 

\section{Preliminaries}\label{S:ss-MO}

In this section, we shall briefly review the notions of \emph{semimodular space} and of \emph{Musielak-Orlicz space}, and moreover, we shall prove some propositions which will be needed later.

\subsection{Semimodular space}

Let $X$ be a (real) vector space. We start with defining semimodular and semimodular space (see, e.g.,~\cite[Definitions 2.1.1 and 2.1.6]{DHHR}).

\begin{df}[Semimodular]
A functional $\rho : X \to [0,+\infty]$ defined on $X$ is called a \emph{semimodular} on $X$, if $\rho$ satisfies {\rm (i)} $\rho(0)=0$\/{\rm ; (ii)} $\rho(\lambda x)=\rho(x)$ for $x \in X$ and $\lambda \in \R$ with $|\lambda| = 1$\/{\rm ; (iii)} $\rho$ is convex on $X$\/{\rm ; (iv)} for any $x \in X$, the function $\lam \mapsto \rho(\lam x)$ is left-continuous at $\lam = 1$\/{\rm ; (v)} if $\rho(\lam x) = 0$ for any $\lam > 0$, then $x=0$. In addition, $\rho$ is called a \emph{modular} on $X$, if $\rho(x)=0$ implies $x = 0$.
\end{df}

\begin{df}[Semimodular space]
Let $\rho$ be a semimodular (modular, respectively) on $X$. Then
\begin{align*}
X_\rho :=& \left\{ x \in X \colon \lim_{\lam \to 0_+} \rho(\lam x) = 0\right\}\\
=& \big\{ x \in X \colon \rho(\lambda x) < +\infty \ \mbox{ for some } \lambda > 0 \big\}
\end{align*}
is called a \emph{semimodular space} (\emph{modular space}, respectively). The norm $\|\cdot\|_\rho$ of $X_\rho$ is given as the \emph{Luxemburg-type norm},
$$
\|x\|_\rho := \inf \left\{ \lam > 0 \colon \rho(x/\lam) \leq 1 \right\} \ \mbox
{ for } \ x \in X_\rho.
$$
\end{df}

\subsection{Musielak-Orlicz space}\label{Ss:MO}

In the rest of this section, let $\Omega$ be an open set in $\R^N$. The class of Musielak-Orlicz spaces is a subclass of the class of semimodular spaces defined above. On the other hand, it includes \emph{variable exponent Lebesgue spaces} $L^{p(x)}(\Omega)$ as well as standard Lebesgue spaces $L^p(\Omega)$ as a special case (see~\cite[Definitions 2.3.1, 2.3.9 and 2.4.4]{DHHR}).

\begin{df}[Generalized $\Phi$-function and $N$-function]\label{D:Phi-f}
Given $\varphi : \Omega \times [0,+\infty) \to [0,+\infty]$, $\varphi$ is called a \emph{generalized $\Phi$-function} (or \emph{Musielak-Orlicz function}) if the following {\rm (i)} and {\rm (ii)} hold\/{\rm :}
\begin{enumerate}
 \item[\rm (i)] for a.e.~$x \in \Omega$, the function $r \mapsto \varphi(x,r)$ is left-continuous and convex on $[0,+\infty)$ and satisfies $\varphi(x,0)=0$, $\lim_{r \to 0_+}\varphi(x,r)=0$ and $\lim_{r \to +\infty} \varphi(x,r) = +\infty$, that is, $\varphi(x,\cdot)$ is a $\Phi$-function (or Orlicz function);
 \item[\rm (ii)] for every $r \geq 0$, the function $x \mapsto \varphi(x,r)$ is measurable in $\Omega$. 
\end{enumerate}
In addition, if $\varphi(x,r)>0$ for all $r > 0$ and a.e.~$x \in \Omega$, then $\varphi$ is said to be \emph{positive}. A generalized $\Phi$-function $\phj$ is further called a \emph{generalized $N$-function} if it is positive and continuous in $r$ and additionally enjoys
\begin{equation}\label{Nfun}
  \lim_{r\to 0_+}\frac{\phj(x,r)}r = 0, \quad
  \lim_{r\to +\infty}\frac{\phj(x,r)}r = +\infty
\quad \mbox{ for a.e. } x \in \Omega.
\end{equation}
\end{df}

Each generalized $\Phi$-function $\phj$ generates a semimodular $\modu: L^0(\Omega) \to [0,+\infty]$ given by
\begin{equation*}
   \modu(v):= \io \phj(x,|v(x)|)\,\dix \ \mbox{ for } v \in L^0(\Omega),
\end{equation*}
where $L^0(\Omega)$ stands for the set of Lebesgue-measurable functions.
In addition, if $\phj$ is positive, then $\modu$ turns out to be a modular (see~\cite[Lemma 2.3.10]{DHHR}).
As in~\cite[Definition 2.3.11]{DHHR}, we define

\begin{df}[Musielak-Orlicz space]
Let $\phj : \Omega \to \R$ be a generalized $\Phi$-function and let $\modu$ is the semimodular associated with $\phj$. Then the semimodular space $X_{\modu}$ is denoted by $L^\phj(\Omega)$ and called a \emph{Musielak-Orlicz space}, whose norm $\|\cdot\|_{L^\phj(\Omega)}$ is given by the Luxemburg-type norm $\|\cdot\|_{\modu}$ of $X_{\modu}$, that is,
\begin{equation*}
  \| v \|_{L^\varphi(\Omega)} = \| v \|_{\modu}
   = \inf\left\{ \lambda>0 \colon \modu(v/\lam) \le 1 \right\}.
\end{equation*}
\end{df}

We further define 
\begin{align*}
E^\phj(\Omega) &:= \big\{ v\in L^\phj(\Omega) \colon \modu(\lambda v) < +\infty \ \mbox{ for all } \lambda > 0 \big\},\\
L^\phj_{OC}(\Omega) &:= \big\{ v\in L^\phj(\Omega) \colon \phj(\cdot,|v(\cdot)|) \in L^1(\Omega) \big\},
\end{align*}
where $L^\phj_{OC}(\Omega)$ is called a \emph{Musielak-Orlicz class} and these sets 
may not coincide in general, but satisfy the relation,
$$
E^\phj(\Omega) \subset L^\phj_{OC}(\Omega) \subset L^\phj(\Omega)
$$
(see~\cite[Definition 2.5.1]{DHHR}). However, under the \emph{$\Delta_2$-condition} of $\phj$ (see~\cite[Definition 2.4.1]{DHHR}) or \eqref{K2} below), they coincide  with  each other. The Musielak-Orlicz space $L^\phj(\Omega)$ is a Banach space (see~\cite[Theorem 2.3.13]{DHHR}), and moreover, $E^\phj(\Omega)$ is a closed subspace of $L^\phj(\Omega)$.

The Musielak-Orlicz space $L^\phj(\Omega)$ has three notions of ``dual'' spaces, that is, the \emph{associate space} $L^\phj(\Omega)'$ and (topological) \emph{dual space} $L^\phj(\Omega)^*$ of $L^\phj(\Omega)$, and moreover, the Musielak-Orlicz space $L^{\phj^*}(\Omega)$ for the \emph{conjugate function} $\phj^*$ of $\phj$ defined as
\begin{equation}\label{phi*}
\phj^*(x,s) := \sup_{r \geq 0} \left( rs - \phj(x,r)\right) \ \mbox{ for } \ x \in \Omega \ \mbox{ and } \ s \geq 0
\end{equation}
(see~\cite[Definition 2.6.1]{DHHR}). Note that $\phj^*$ is an $N$-function (a generalized $\Phi$-function) if so is $\phj$ (see~\cite[Lemma 2.5.8]{HH},~\cite[Theorem 2.6.8]{DHHR}). Here, the associate space is given by
$$
L^\phj(\Omega)' = \left\{ v \in L^0(\Omega) \colon \|v\|_{L^\phj(\Omega)'} < +\infty\right\}
$$
equipped with norm $\|v\|_{L^\phj(\Omega)'} := \sup \{\int_\Omega |u||v| \, \dix \colon u \in L^\phj(\Omega), \ \|u\|_{L^\phj(\Omega)} \leq 1\}$ for $v \in L^0(\Omega)$ (see~\cite[Definition 2.7.1]{DHHR}). 
Then all these spaces are Banach spaces and enjoy the following relations:
\begin{equation}\label{dual_spaces}
L^{\phj^*}(\Omega) \subset L^\phj(\Omega)' \subset L^\phj(\Omega)^*.
\end{equation}
We also recall the \emph{H\"older inequality} (see~\cite[Lemma 2.6.5]{DHHR}),
\begin{equation}\label{hoelder}
 \int_\Omega |uv| \, \d x \leq 2 \|u\|_{L^\phj(\Omega)} \|v\|_{L^{\phj^*}(\Omega)} \quad \mbox{ for } \ u \in L^\phj(\Omega), \ v \in L^{\phj^*}(\Omega),
\end{equation}
where the constant $2$ cannot be omitted. Here and henceforth, each function $v \in L^{\phj^*}(\Omega)$ may be identified with a bounded linear functional $J_v \in L^\phj(\Omega)^*$ given by 
$$
J_v : u \in L^{\phj}(\Omega) \mapsto J_v(u) := \int_\Omega u(x)v(x) \, \d x,
$$ 
and $J_v$ will also be denoted by $v$, when no confusion can arise (see \S \ref{Ss:dual} below). Here, we stress that $L^{\phj^*}(\Omega)$ cannot be \emph{isometrically} identified with (a subset of) $L^\phj(\Omega)^*$, and hence, the H\"older inequality \eqref{hoelder} seems different from the Schwarz inequality between $L^{\phj}(\Omega)$ and its dual $L^\phj(\Omega)^*$. However, we shall not distinguish the norms of $L^{\phj^*}(\Omega)$ and $L^\phj(\Omega)^*$, unless any serious confusion may arise. Moreover, let us recall \emph{Young's inequality},
\begin{equation}\label{phi-young}
rs \leq \phj(x,r) + \phj^*(x,s) \ \mbox{ for } \ r,s \geq 0 \ \mbox{ and } \ x \in \Omega,
\end{equation}
which follows immediately from \eqref{phi*}. We shall further discuss for which modulars the three spaces of \eqref{dual_spaces} coincide (see \S \ref{Ss:dual} below).

The following fact will be often used throughout the paper; so we give a statement with a proof for the convenience of the reader, although it is well known.

\begin{prop}\label{P:nondec}
Let $\phj : \R \to [0,+\infty]$ be an even and convex function satisfying $\phj(0)=0$. Then for any $r \in \R$ 
the function $\lam \in [0,+\infty) \mapsto \phj(\lam r)$ is non-decreasing. In addition, if $\phj$ is positive,
then the function $\lambda \mapsto \varphi(\lambda r)$ is strictly increasing for $r \neq 0$, whenever $\phj(\lam r)$ is finite.
\end{prop}

\begin{proof}
Let $0 \leq \lam < \lam' < +\infty$. It then follows that
\begin{align*}
\phj(\lam r) = \phj \left( \frac \lam {\lam'} (\lam' r) + \frac{\lam'-\lam}{\lam'} \,0\right) \leq \frac \lam {\lam'} \phj(\lam'r) \leq \phj(\lam' r) \ \mbox{ for } r \in \R.
\end{align*}
If $\phj$ is positive, i.e., $\phj (r) > 0$ for $r \neq 0$, then $\phj(\lam r) < \phj(\lam' r)$ 
for $r \neq 0$ and $\lambda' > \lambda$. 
\end{proof}

\subsection{Characterization of dual spaces}\label{Ss:dual}

This subsection is devoted to discussing characterization and useful properties of topological dual of Musielak-Orlicz spaces. An important conclusion of this subsection is a variant of the Riesz representation theorem: under certain assumptions (see (v) of Proposition \ref{P:proper} for more details), the topological dual $L^\phj(\Omega)^*$ of $\Lpsot$ can be identified with the Musielak-Orlicz space associated with $\varphi^*$, which is the conjugate function of $\varphi$, that is,
\begin{equation*}
L^\phj(\Omega)^* \simeq L^{\phj^*}(\Omega).
\end{equation*}
In what follows, a generalized $\Phi$-function $\phj = \phj(x,r)$ is said to be \emph{proper} in $\Omega$ (see~\cite[Definition 2.7.8]{DHHR}), if it holds that
\begin{equation*}
S(\Omega) \subset L^\phj(\Omega) \cap L^\phj(\Omega)',
\end{equation*}
where $S(\Omega)$ denotes the set of simple functions (i.e., a finite linear combination of characteristic functions supported over measurable sets of finite measure) defined in $\Omega$. Moreover, $\varphi$ is said to be \emph{locally integrable} in $\Omega$, if for any $\lam > 0$ and measurable set $\omega \subset \Omega$ of finite measure, it holds that
\begin{equation*}
\modu(\lam \chi_\omega) < +\infty,
\end{equation*}
where $\chi_\omega$ stands for the characteristic function supported over $\omega$ 
(see~\cite[Definition 2.5.5]{DHHR}). We need the following
\begin{prop}[Cf.~{\cite[Corollary 2.7.9, Theorems 2.7.4 and 2.7.14]{DHHR}}]
\label{P:proper}
Assume that $\phj$ is a generalized $\Phi$-function on an open set $\Omega$ in $\R^N$. Then the following three conditions are equivalent\/{\rm :} 
\begin{enumerate}
 \item[\rm(i)] $\phj$ is proper in $\Omega$,
 \item[\rm(ii)] $\phj^*$ is proper in $\Omega$,
 \item[\rm(iii)] $S(\Omega) \subset L^\phj(\Omega) \cap L^{\phj^*}(\Omega)$.
\end{enumerate}
Moreover, the following properties are satisfied\/{\rm :}
\begin{enumerate}
 \item[\rm (iv)] If $S(\Omega) \subset L^\phj(\Omega)$, then the associate space $L^\phj(\Omega)'$ coincides with $L^{\phj^*}(\Omega)$. Moreover, it holds that
\begin{equation*}
 \|v\|_{L^{\phj^*}(\Omega)} \leq \|v\|_{L^\phj(\Omega)'} = \|v\|_{L^\phj(\Omega)^*} \leq 2 \|v\|_{L^{\phj^*}(\Omega)} \ \mbox{ for } \ v \in L^{\phj^*}(\Omega).
\end{equation*}
Moreover, it holds that
\begin{align*}
 \modud(v) &= (\modu)^*(v)\\
 &:= \sup_{u \in L^\varphi(\Omega)} \left\{ \langle v,u \rangle_{L^\varphi(\Omega)} - \modu(u) \right\} \ \mbox{ for all } \ v \in L^{\phj^*}(\Omega).
\end{align*}
 \item[\rm (v)] If $\phj$ is proper and locally integrable on $\Omega$ and $E^\phj(\Omega) = L^\phj(\Omega)$, then the {\rm (}topological\/{\rm )} dual space $L^{\phj}(\Omega)^*$ of $L^\phj(\Omega)$ is isomorphic {\rm (}but not isometric{\rm )} to $L^{\phj^*}(\Omega)$.
 \item[\rm (vi)] 
If $\phj^*$ is locally integrable on $\Omega$, then there is a function $\rho=\rho(\omega)$ defined for Lebesgue measurable subsets $\omega \subset \Omega$ of finite measure with values in $[0,+\infty)$ such that $\rho(\omega) \to 0_+$ as $|\omega|\to0_+$ and 
$$
\int_\omega |f| \, \d x \leq \rho(\omega) \|f\|_{L^\phj(\Omega)}
$$
for all $f \in L^\phj(\Omega)$ and Lebesgue measurable subsets $\omega \subset \Omega$. 
  \item[\rm (vii)] Under the same assumptions as in {\rm (v)} {\rm (}i.e., $L^{\varphi}(\Omega)^* \simeq L^{\varphi^*}(\Omega)$\/{\rm )}, the functional $\modud$ coincides with the \emph{convex conjugate} $(\modu)^*$ of $\modu$ on $L^\phj(\Omega)^*$. Moreover, it holds that
$$
(\modud)^*(u) = (\modu)^{**}(u) = \modu(u) \ \mbox{ for  } \ u \in L^\varphi(\Omega).
$$
\item[\rm (viii)] Suppose that $\Omega$ has a finite measure and $T > 0$ is finite. If $\phj$ is proper and locally integrable in $\Omega$, then $\phj = \phj(x,r)$ is also proper and locally integrable in $Q = \Omega \times (0,T)$.
\end{enumerate}
\end{prop}

\begin{proof}
The equivalence among (i)--(iii) is proved in~\cite[Corollary 2.7.9]{DHHR}. We refer the reader to~\cite[Theorems 2.7.4 and 2.7.14]{DHHR} for the proof of (iv) and (v), respectively (see also~\cite[Remark 2.7.16]{DHHR}). Moreover, (vi) follows immediately from the observation,  by H\"older's  inequality \eqref{hoelder},
    $$
    \int_\omega |f| \, \d x= \int_\Omega \chi_\omega |f| \, \d x \leq 2\|f\|_{L^\phj(\Omega)} \|\chi_\omega\|_{L^{\phj^*}(\Omega)}
    $$
    for any $f \in L^{\phj}(\Omega)$ and measurable subsets $\omega \subset \Omega$. Here we note that $\rho(\omega):=2\|\chi_\omega\|_{L^{\phj^*}(\Omega)} \to 0$ as $|\omega| \to 0$, since $\phj^*$ is locally integrable on $\Omega$ (see~\cite[Proposition 2.5.7]{DHHR}). As for (vii), noting by assumption that $S(\Omega) \subset L^\phj(\Omega)$, we have already seen that 
    $(\modu)^*=\modud$ on $L^{\phj^*}(\Omega)$ by (iv) and $L^{\phj^*}(\Omega) = L^\phj(\Omega)^*$ by (v). Hence it follows that $(\modud)^*=\modu$. Finally, we prove (viii). Let $\lam > 0$ and let $E$ be a measurable subset in $Q$. Set $E_t := \{x \in \Omega \colon (x,t) \in E \} \subset \Omega$. Since $\phj$ is locally integrable in $\Omega$ and $\Omega$ has a finite measure, with the aid of Fubini's theorem, we infer that
    $$
    \iint_{Q} \phj(x,\lam\chi_E(x,t)) \, \dix \dit
    = \int^T_0 \left( \int_{E_t} \phj(x,\lam) \, \dix \right) \dit
    \leq T \int_\Omega \phj(x,\lam) \, \dix < + \infty.
    $$
Hence $\phj(\cdot,\lam \chi_E(\cdot,\cdot))$ belongs to $L^1(Q)$, and therefore, $\phj$ is locally integrable in $Q$.
We can also prove that $\phj$ is proper in $Q$ by repeating a similar argument.
\end{proof}

Finally, we shall discuss the reflexivity and separability of Musielak-Orlicz spaces.
\begin{prop}[Reflexivity and separability of Musielak-Orlicz spaces]\label{P:reflexive}
Let $\phj$ be a proper generalized $\Phi$-function in $\Omega$ such that $\phj$ and $\phj^*$ are locally integrable. 
Suppose also that
\begin{equation}\label{E=L}
E^\phj(\Omega) = L^\phj(\Omega), \quad E^{\phj^*}(\Omega) = L^{\phj^*}(\Omega). 
\end{equation}
Then $L^\phj(\Omega)$ is reflexive and separable {\rm (}so does the dual space{\rm )}.
\end{prop}

We refer the reader to~\cite[Lemma 2.7.17, Corollary 2.7.18 and Theorem 2.5.10]{DHHR} for a proof.
Moreover, we note that \eqref{E=L} can be checked when $\phj$ and $\phj^*$ fulfill the $\Delta_2$-condition (see Lemma \ref{L:MO-chk}).

\section{Assumptions and main result}
\label{sec:main}

As a more precise formulation of \eqref{proto-eqn}, we shall consider
\begin{equation}\label{eqn}
  \alpha(x,u_t(x,t)) + b(x,t) = f(x,t), \quad b(t) \in B(u(t)) \ \text{ for } \ x \in \Omega, \ t \in (0,T),
\end{equation}
where $u(t) := u(\cdot,t)$ and $B$ is a (possibly) multi-valued operator from a certain Banach space $V$ into the dual space $V^*$; hence $B(u(t))$ is a subset of $V^*$ and $b(t) = b(\cdot,t)$ is one of its elements. In this section, we shall present basic assumptions on the nonlinear terms involved in equation \eqref{eqn} and also provide some remarks and related propositions for later use. They will permit us to set up a functional analytic framework for the problem. Finally, the main result of the present paper will be stated at the end of the section. In the rest of the paper, we always assume that $\Omega$ is an open set in $\R^N$ of finite measure, i.e., $|\Omega|<+\infty$.

\subsection{Assumptions for $\alpha(x,r)$}

Let us start with giving assumptions for the nonlinear operator acting on the time-derivative.
\begin{assu}[Hypotheses on $\alpha$ and $\phj$]\label{hp:alpha}{\rm
The operator $\alpha : \Omega \times \R \to [0,+\infty)$ is \emph{single-valued} and written  in  the form
$$
\alpha(x,r) = \partial \phj(x,r),
$$
where $\partial \phj$ stands for the \emph{subdifferential} with respect to the second variable $r$ of a function $\phj:\Omega \times \RR \to [0,+\infty)$, that is,
$$
\partial \phj(x,r) := \left\{ \xi \in \R \colon \phj(x,r') - \phj(x,r) \geq \xi (r' - r) \ \mbox{ for any } \, r' \in \R \right\}
$$
for a.e.~$x \in \Omega$ and $r \in \R$. Let $\phj^* : \Omega \times \R \to (-\infty,+\infty]$ be the \emph{convex conjugate} of $\phj$ with respect to the second variable, that is,
\begin{equation}\label{cc}
\phj^*(x,r) := \sup_{s \in \R} \left\{ rs - \varphi(x,s)\right\}
\ \mbox{ for } \ x \in \Omega \ \mbox{ and } \ r \in \R.
\end{equation}
The following conditions (a)--(c) hold:
\begin{enumerate}
 \item[(a)] For a.e.~$x\in \Omega$, $\phj(x,\cdot)$ is strictly convex, lower semicontinuous and even (i.e., $\phj(x,r) = \phj(x,|r|)$ for $r \in \R$), and moreover, $\partial \phj(x,0) = 0$.
 \item[(b)] \begin{enumerate}
	     \item[(i)] For all~$r\in \RR$, $\phj(\cdot,r)$ is measurable in $\Omega$. 
	     \item[(ii)] It holds that $\phj(\cdot,1), \phj^*(\cdot,1) \in L^1(\Omega)$.
	    \end{enumerate}
\item[(c)] Both $\phj$ and $\phj^*$ satisfy the \emph{$\Delta_2$-condition}
uniformly in $\Omega$, i.e., there exist ${K} \ge 2$
such that 
\begin{equation}\label{K2}
 \phj(x,2r)\le K \phj(x,r), \quad
  \phj^*(x,2r)\le {K} \phj^*(x,r)
\end{equation}
for all $r\in [0,+\infty)$ and a.e.~$x \in \Omega$.
\end{enumerate}
}\end{assu}

Here $\phj$ is also said to satisfy the \emph{$\nabla_2$-condition} if $\phj^*$ satisfies the $\Delta_2$-condition (see, e.g.,~\cite{HH}). Hence the condition (c) above means that $\phj$ satisfies both the $\Delta_2$- and the $\nabla_2$-conditions.

Let us here give several remarks. 
\begin{enumerate}
 \item[(i)] In what follows, due to the symmetry, we may often regard $\phj(x,\cdot)$ as a function defined only on the half-line $[0,+\infty)$ and identify $\phj(x,r)$ with $\phj(x,|r|)$ for $r \in \R$. Indeed, such a setting is more consistent with the theory of Musielak-Orlicz space (see \S \ref{Ss:MO} and~\cite{DHHR}). On the other hand, the original setting is fitter to the subdifferential calculus. So we shall not change the present notation. The convex conjugate defined in~\eqref{cc} coincides with the notion of the conjugate generalized $\Phi$-function defined by \eqref{phi*} for any $r \geq 0$, and therefore, we do not distinguish the notation. 
 \item[(ii)] It is well known (see, e.g., the monographs \cite{barbu,Br}) that subdifferentials form a subclass of maximal monotone graphs, and moreover, in general, they can be multi-valued. However, for the sake of simplicity, throughout the paper, we assume that $\alpha(x,\cdot) = \partial \varphi(x,\cdot)$ is single-valued for almost every $x\in \Omega$. Note that this fact guarantees that $r\mapsto \phj(x,r)$ is differentiable (and therefore, continuous) on $\RR$ for almost every $x\in \Omega$. Moreover, under Assumption \ref{hp:alpha}, $r\mapsto \alpha(x,r)$ is continuous in $\R$ due to the maximality of $\alpha(x,\cdot)$ and the $\Delta_2$-condition (see Remark \ref{R:al-conti} below for a proof). Hence one can assume $\phj$ to be of class $C^1$ (instead of the lower semicontinuity) in the second variable in Assumption \ref{hp:alpha} without loss of generality. 
 \item[(iii)] It also follows from~$(a)$ that $\alpha(x,\cdot)$ is odd, and $\phj = \phj(x,r)$ can be supposed to be positive (see Definition \ref{D:Phi-f}) and vanish at $r=0$ without loss of generality. Hence in order to check $\phj$ being a generalized $\Phi$-function, it only remains to prove $\lim_{r \to 0_+}\varphi(x,r)=0$ and $\lim_{r \to +\infty} \varphi(x,r) = +\infty$, which will be discussed in Lemma \ref{L:Nfun} below. Moreover, (ii) of (b) will turn out to be equivalent to the property that $\phj$ is proper in $\Omega$ under the present setting (see Lemma \ref{L:MO-chk} and Remark \ref{R:MO-chk} below for details).
\end{enumerate}

Moreover, we have
\begin{lem}\label{L:phj*}
Let {\rm Assumption \ref{hp:alpha}} hold. Then the conjugate function $\phj^*(x,r)$ is finite for $r \in \R$ and for a.e.~$x \in \Omega$. Moreover, $\phj^*$ satisfies all the conditions in {\rm Assumption \ref{hp:alpha}} for $\phj$.
\end{lem}

\begin{proof}
For a.e.~$x \in \Omega$, $\phj^*(x,\cdot)$ is lower semicontinuous and convex in $\R$ (see, e.g.,~\cite[\S 1.4]{FA}) and, for $r \in \R$, $\phj^*(\cdot,r)$ is measurable in $\Omega$ due to (i) of (b) in Assumption \ref{hp:alpha} (for $\phj$). Moreover, under (a), one can check that $\varphi^*(x,0)=0$ and $\varphi^*(x,\cdot)$ is non-negative and even. The condition {(c)} along with (a) implies $\phj^*(x,r) < +\infty$ for any $r \geq 0$ and a.e.~$x \in \Omega$ (hence $\phj^*(x,\cdot)$ is continuous on $[0,+\infty)$ for a.e.~$x \in \Omega$). Indeed, we immediately have an alternative: either $\phj^*$ is finite everywhere, or $\phj^* = I_{\{0\}}$, i.e., 
the indicator function supported on the origin $\{0\}$. However, the latter condition implies $\phj(x,\cdot) = \phj^{**}(x,\cdot) \equiv 0$, which contradicts the positivity of $\phj(x,\cdot)$. Furthermore, {(c)} also ensures that $\phj^*(x,\cdot)$ is positive, that is, $\phj^*(x,r) > 0$ for any $r > 0$; otherwise, we have $\phj^*(x,\cdot) \equiv 0$ by {(c)} (see also Proposition \ref{P:nondec}). However, it implies $\phj(x,\cdot) = \phj^{**}(x,\cdot) \equiv I_{\{0\}}$, which is a contradiction to (a). Furthermore, since $\phj$ is strictly convex (and of class $C^1$ in $\R$), we find that $\alpha(x,\cdot)$ is strictly increasing and continuous in $\R$. Therefore the inverse $\alpha^{-1}(x,\cdot) = \partial \phj^*(x,\cdot)$ of $\alpha(x,\cdot)$ is also single-valued, strictly increasing and continuous in $\R$. Hence $\phj^*(x,\cdot)$ turns out to be strictly convex and of class $C^1$ in $\R$. Moreover, we see that $\partial \phj^*(x,0) = 0$.
\end{proof}

\beos\label{single}
It seems possible to extend most of results in the present paper to the case where $\alpha(x,\cdot)$ is a multi-valued operator. On the other hand, this extension may involve a quite relevant  amount of additional technical work. In view of the fact that the present setting is already rather complicated, we prefer to  focus  on the single-valued case only.  
\eddos

Now, we shall discuss some important consequences of Assumption \ref{hp:alpha},
which will be used in order to properly formulate the main result of the present paper. In particular, we observe that conditions (a) and {(c)} imply a superlinear growth both of $\varphi(x,\cdot)$ and the conjugate $\varphi^*(x,\cdot)$; indeed, the constant $K_0$ in the following lemma will be chosen to be \emph{strictly} greater than $2$.
\begin{lem}\label{lemma:1x} 
 Let\/ {\rm (a)} and {\rm {(c)}} of {\rm Assumption~\ref{hp:alpha}} hold. Then there exists $K_0 > 2$ such that 
 \begin{equation}\label{jlow}
   \phj(x,2r) \ge K_0 \phj(x,r), \quad
   \phj^*(x,2s) \ge K_0 \phj^*(x,s)
 \end{equation}
for all $r,s \ge 0$ and a.e.~$x \in \Omega$.
\end{lem}
\begin{proof}
We shall first prove the assertion for $\varphi$. Let $r > 0$ and let $s = \alpha(x,r) \in \partial \varphi(x,r)$ for a generic $x \in \Omega$. Thanks to the definition of subdifferential, we have
\begin{equation}\label{low21}
  \phj(x,2r) \ge \phj(x,r) + s ( 2r - r )
   = \phj(x,r) + rs.
\end{equation}  
On the other hand, the Fenchel-Moreau identity reads,
\begin{equation}\label{low22}
  rs = \phj(x,r) + \phj^*(x,s).
\end{equation}  
Combining these facts, we obtain
\begin{equation}\label{low23}
  \phj(x,2r) \ge 2 \phj(x,r) + \phj^*(x,s).
\end{equation}  
Now, noting that $r \in \partial \varphi^*(x,s)$ (indeed, $\partial \varphi^*(x,\cdot)$ coincides with the inverse map of $\partial \varphi(x,\cdot)$), we derive analogously to \eqref{low21} that
\begin{equation*}
  \phj^*(x,s) \le \phj^*(x,2s) + r ( s - 2s )
   = \phj^*(x,2s) - rs,
\end{equation*}  
whence, by \eqref{K2},
\begin{equation*}
  \phj^*(x,s) \le {K} \phj^*(x,s) - rs.
\end{equation*}  
Consequently, invoking \eqref{low22} again, one observes that
\begin{equation}\label{low26}
  \phj^*(x,s) \ge \frac{rs}{{K} - 1}
   = \frac1{{K} - 1} \left[ \phj(x,r) + \phj^*(x,s) \right].
\end{equation}  
Substituting \eqref{low26} into \eqref{low23}, we infer that
\begin{equation*}
  \phj(x,2r) \ge \left( 2 + \frac1{{K} - 2} \right) \phj(x,r),
\end{equation*}  
whence the assertion follows. The assertion for $\varphi^*$ can be verified in the same manner. 
\end{proof}

For later use (see \S \ref{sec:ex}), we remark that the first inequality of \eqref{jlow} also implies the $\Delta_2$-condition for $\phj^*$.
\begin{lem}\label{L:Del2-inv}
Let {\rm (a)} of {\rm Assumption \ref{hp:alpha}} hold and assume that there exists a constant $K_0 > 2$ such that
$$
\phj(x,2r) \geq K_0 \phj(x,r) \quad \mbox{ for } \ r \geq 0 \ \mbox{ and a.e. }  x \in \Omega.
$$
Then $\phj^*$ satisfies the $\Delta_2$-condition.
\end{lem}

\begin{proof}
By assumption, it follows that
\begin{align*}
 \phj^*(x,r) &= \sup_{\rho \geq 0} \left\{ r\rho - \phj(x,\rho) \right\}\\
&\geq \sup_{\rho > 0} \left\{ r\rho - K_0^{-1} \phj(x,2\rho) \right\}\\
&= K_0^{-1} \sup_{\rho > 0} \left\{(K_0/2)r \cdot 2\rho - \phj(x,2\rho)\right\}
= K_0^{-1} \phj^*(x,(K_0/2)r)
\end{align*}
for $r \geq 0$ and a.e.~$x \in \Omega$. Since $K_0/2 > 1$, one can take $n_0 \in \N$ such that $(K_0/2)^{n_0-1} < 2 \leq (K_0/2)^{n_0}$, and therefore, it follows from Proposition \ref{P:nondec} that
$$
\phj^*(x,2r) \leq \phj^*(x,(K_0/2)^{n_0} r) \leq K_0^{n_0} \phj^*(x,r),
$$
which is the $\Delta_2$-condition for $\phj^*$ with $K = K_0^{n_0} > 2$.
\end{proof}

In the next lemma, $\varphi(x,\cdot)$ turns out to be a generalized $\Phi$-function (and moreover, $N$-function). Thus we can define the Musielak-Orlicz spaces associated with $\phj$ and $\phj^*$.

\begin{lem}\label{L:Nfun}
Let {\rm (a)} and {\rm {(c)}} of {\rm Assumption~\ref{hp:alpha}} hold. Then \eqref{Nfun} holds. Hence, $\phj$ meets the requirements for being a generalized $N$-function under {\rm (i)} of \/{\rm (b)} in {\rm Assumption~\ref{hp:alpha}}, and in particular, $\phj$ is a generalized $\Phi$-function. Furthermore, \eqref{Nfun} also holds with $\phj$ replaced by $\phj^*$. 
\end{lem}

\begin{proof}
For $r > 1$, one can take $n \in \N$ such that $2^{n-1} < r \leq 2^n$. Hence it follows that
$$
\frac{\varphi(x,r)}{r} \geq \frac{\varphi(x,2^{n-1})}{2^n} \stackrel{\eqref{jlow}}\geq \frac{K_0^{n-1} \varphi(x,1)}{2^n} \to +\infty
$$
as $n \to +\infty$ (equivalently, $r \to +\infty$). Here we used the fact that $\varphi(x,1) > 0$ and $K_0 > 2$. Moreover, for $r \in (0,1)$, let $n \in \N$ be such that $2^{-(n+1)} < r \leq 2^{-n}$. We then also derive that
$$
\frac{\varphi(x,r)}{r} \leq \frac{\varphi(x,2^{-n})}{2^{-(n+1)}} \stackrel{\eqref{jlow}} \leq \frac{K_0^{-n} \varphi(x,1)}{2^{-(n+1)}} \to 0
$$
as $n \to +\infty$ (equivalently, $r \to 0_+$). Combining these facts along with (a) and (i) of (b) in Assumption~\ref{hp:alpha}, $\varphi$ turns out to be a generalized $N$-function (hence, it is in particular a generalized $\Phi$-function). Finally, \eqref{Nfun} for $\varphi^*(x,r)$ can be proved in the same manner.
\end{proof}

\begin{rem}[Continuity of $\alpha(x,\cdot)$]\label{R:al-conti}
Let us give a proof for the continuity of $r \mapsto \alpha(x,r)$ in $\R$ for a.e.~$x \in \Omega$ under the present setting (see a remark just below Assumption \ref{hp:alpha}). Let $r_n \to r$ in $\R$. Then the Fenchel-Moreau identity implies
$$
\phj(x,r_n) + \phj^*(x,\alpha(x,r_n)) = r_n \alpha(x,r_n).
$$
Since $\phj^*$ is  a generalized  $N$-function (i.e., $\phj^*$ is coercive) by Lemma \ref{L:Nfun} and $\phj$ is nonnegative, $(\alpha(x,r_n))$ turns out to be bounded in $\R$ for each $x \in \Omega$. Hence one can extract a (not relabeled) subsequence of $(n)$ such that $\alpha(x,r_n) \to a$ for some $a \in \R$. Due to the (demi)closedness of $\alpha(x,\cdot)$ in $\R \times \R$, we obtain $a = \alpha(x,r)$. Furthermore, by virtue of the uniqueness of the limit, we conclude that $\alpha(x,r_n) \to \alpha(x,r)$ without taking any subsequence. Thus $\alpha(x,\cdot)$ turns out to be continuous, that is, $\phj(x,\cdot)$ is of class $C^1$ in $\R$.
\end{rem}

We shall next check that, under the frame of Assumption \ref{hp:alpha}, the dual space $L^\phj(\Omega)^*$ is isomorphic to the Musielak-Orlicz space $L^{\phj^*}(\Omega)$.
\begin{lem}\label{L:MO-chk}
In addition to {\rm (a)} and {\rm (b)} of {\rm Assumption \ref{hp:alpha}}, assume that $\phj$ satisfies the $\Delta_2$-condition, which is a part of {\rm {(c)}} in {\rm Assumption \ref{hp:alpha}}. Then $\phj$ is proper and locally integrable in $\Omega$. Moreover, $E^\phj(\Omega)$ coincides with $L^\phj(\Omega)$, and hence, $L^\phj(\Omega)^*$ is isomorphic to $L^{\phj^*}(\Omega)$. Furthermore, $L^\phj(Q)^*$ is isomorphic to $L^{\phj^*}(Q)$. 
\end{lem}

\begin{proof}
For any $e \in S(\Omega)$ we see that $x \mapsto \phj(x,e(x))$ is measurable in $\Omega$, since $\phj$ is Carath\'eodory due to (a) of Assumption \ref{hp:alpha}. Since $\phj(\cdot,1)$ is assumed to be integrable in $\Omega$ (see (ii) of (b) in Assumption \ref{hp:alpha}), by Proposition \ref{P:nondec} we find that $\phj(\cdot,e) \in L^1(\Omega)$. Similarly, it also follows that $\phj^*(\cdot,e) \in L^1(\Omega)$. Hence $e \in L^\phj(\Omega) \cap L^{\phj^*}(\Omega)$. Recalling $L^{\phj^*}(\Omega) \subset L^\phj(\Omega)'$, we deduce that $e \in L^{\phj}(\Omega)'$, and thus, $\phj$ turns out to be proper.

Moreover, if $\phj$ satisfies the $\Delta_2$-condition, then $\phj$ turns out to be locally integrable in $\Omega$. Indeed, let $u \in L^\phj(\Omega)$ and $\lam > 1$. Then one can take $n \in \N$ such that $2^{n-1} < \lam \leq 2^n$. It follows from Proposition \ref{P:nondec} and the $\Delta_2$-condition of $\phj$ (see \eqref{K2}) that
$$
\phj(x,\lam u(x)) \leq \phj(x,2^n u(x)) \leq K^{n} \phj(x,u(x)),
$$
which implies
$$
\modu(\lam u) \leq K^n \modu(u) < +\infty.
$$
On the other hand, one can check $\modu(\lam u) \leq \lam \modu(u) < +\infty$ for any $\lam \in (0,1)$ by convexity. In particular, since $\phj$ is proper, we have $\modu(\lam \chi_\omega) < +\infty$ for any $\lam > 0$ and measurable $\omega \subset \Omega$ (of finite measure). Moreover, the $\Delta_2$-condition of $\phj$ also ensures $E^\phj(\Omega)=L^\phj(\Omega)$ (see~\cite[p.~49]{DHHR}). Consequently, under (a) and (b) of Assumptions \ref{hp:alpha} along with the $\Delta_2$-condition of $\phj$, $L^\phj(\Omega)^*$ is isomorphic to $L^{\phj^*}(\Omega)$. 

Finally, we note that the argument above can also be applied to check these properties in $Q=\Omega\times(0,T)$. Therefore, $L^\phj(Q)^*$ turns out to be isomorphic to $L^{\phj^*}(Q)$.
\end{proof}

\begin{rem}\label{R:MO-chk}
Since $\Omega$ has a finite measure, any proper $\Phi$-function $\phj$ satisfies (b) of Assumption \ref{hp:alpha}. Indeed, if $\phj$ is proper in $\Omega$, then it can be seen that $\phj(\cdot,1) = \phj(\cdot,\chi_\Omega) \in L^1(\Omega)$. Furthermore, thanks to Proposition \ref{P:proper}, $\phj^*$ is also proper, and hence, $\phj^*(\cdot,1)$ is integrable in $\Omega$.
\end{rem}

\subsection{Base spaces and subdifferentials}\label{Ss:base_sp}

We are now ready to introduce a functional analytic framework for handling equation \eqref{eqn}. We first set 
$$
V = \Lpsot
$$ 
as a base space. Furthermore, denote by $V^*$ the dual space of $V$, which can be identified with $L^{\phj^*}(\Omega)$ (see Proposition \ref{P:proper}); hence, in what follows, we shall use the same notation for $f \in V^*$ and its representation in $L^{\phj^*}(\Omega)$. The duality pairing between $\Vp$ and $V$ will be simply denoted by $\duav{\cdot,\cdot}$. We further recall that $V$ and $V^*$ are reflexive under Assumption \ref{hp:alpha} (see Proposition \ref{P:reflexive}). 

We shall need an analogous space of (time-dependent) functions defined on $Q:=\Omega\times(0,T)$. Set
\begin{equation*}
  \calV = \Lpsq:= 
   \big\{ v\in L^0(Q) \colon \Modu(\lam v) < +\infty \ \mbox{ for some } \lam > 0 \big\},
\end{equation*}
where $\Modu: L^0(Q) \to [0,+\infty]$ on $L^0(Q)$ is the modular defined by
\begin{equation*}
 \Modu(v):= \iTo \phj(x,|v(x,t)|)\,\dix\dit \quad \mbox{ for } \ v \in L^0(Q),
\end{equation*}
and which is furnished with norm $\| \cdot \|_{\calV} := \| \cdot \|_{L^\phj(Q)}$. 
Then $\calV$ will also play the role of a base space. Here for 
any $v \in \Lpsq$, by Fubini-Tonelli's lemma, we remark that $v(t):=v(\cdot,t) \in \Lpso$ for a.e.~$t \in (0,T)$. 
We shall denote by $\calV^*$ the dual space of $\calV$, which can be identified with $L^{\phj^*}(Q)$; hence, in what follows, we shall use the same notation for $f \in \calV^*$ and its representation in $L^{\phj^*}(Q)$. Moreover, $\duaV{\cdot,\cdot}$ stands for the duality pairing between $\calVp$ and $\calV$. Under the $\Delta_2$-condition of $\phj$ and $\phj^*$, we also have $\calV = E^\phj(Q)$ and $\calV^* = E^{\phj^*}(Q)$, and therefore, $\calV$ and $\calV^*$ are reflexive and separable (see Propositions \ref{P:proper} and \ref{P:reflexive} and Lemmas \ref{L:phj*} and \ref{L:MO-chk}).

In the sequel, we shall handle various types of monotone operators constructed as subdifferentials. To this end, we shall introduce some general notation. Let $S : V\to (-\infty,+\infty]$ be a convex lower semicontinuous functional such that $S \not\equiv +\infty$. Then we denote by $\de_\Omega S$ the subdifferential of the functional $S$ in the duality between $V$ and $\Vp$. Namely,
\begin{equation*}
  \xi \in \deo S(u) \quad \stackrel{\text{define}}\Longleftrightarrow \quad
   S(v) \ge S(u) + \duav{\xi, v - u } \ \mbox{ for all } v\in V.
\end{equation*}
Then, $\partial_\Omega S$ is a maximal monotone operator from $V$ to $2^{\Vp}$.
Analogously, under similar assumptions on a functional $\calS:\calV\to (-\infty,+\infty]$, we can set
\begin{equation*}
  \xi \in \deq \calS(u) \quad \stackrel{\text{define}}\Longleftrightarrow \quad
   \calS(v) \ge \calS(u) + \duaV{\xi, v - u } \ \mbox{ for all } v\in \calV.
\end{equation*}
Clearly, in the above relation, $\xi\in \calVp$ depends both on 
space and time variables. 

\subsection{Assumptions on $B$}

Let us move on to specifying the second nonlinear
operator involved in \eqref{eqn}.
\begin{assu}[Hypotheses on $B$ and $\fhi$]\label{hp:B}\begin{rm}
The operator $B : V \to 2^{V^*}$ is in a subdifferential form,
$$
B = \partial_\Omega \fhi
$$
for some convex functional $\fhi : V \to [0,+\infty]$. The {\sl effective domain} of $\fhi$ is defined as the set
\begin{equation*}
  D(\fhi):=\big\{ u\in V \colon \fhi(u)<+\infty \big\}.
\end{equation*}
Assume the following:
\begin{itemize}
 \item[(a)] $\fhi$ is convex, lower semicontinuous and not identically $+\infty$ (i.e., $D(\fhi) \neq \emptyset$). 
 \item[(b)] There exists a Banach space $X$ \emph{compactly} embedded into $V$ such that $D(\fhi) \subset X$. Moreover, sublevel sets of the sum $(\fhi+\modu)$ are bounded in $X$; namely, there exists a non-decreasing function $\fQ:[0,+\infty) \to [0,+\infty)$ such that, for all $c\ge 0$, the following holds:
\begin{equation*}
\| u \|_X \le \fQ(c) \ \mbox{ if } \ u\in V \ \mbox{ and } \ \fhi(u)+\modu(u)\le c.
\end{equation*}
\end{itemize}
\eddas
As before, we also employ notation related to 
time-dependent functions. We start with setting $\Fhi:L^1(0,T;V) \to [0,+\infty]$ by
\begin{equation*}
\Fhi(u):= \begin{cases}
	   \int_0^T \fhi(u(t)) \,\dit &\mbox{ if } \ t \mapsto \fhi(u(t)) \in L^1(0,T),\\
	   +\infty &\mbox{ otherwise}
	  \end{cases}
\end{equation*}
for $u \in L^1(0,T;V)$.
Then the following lemma holds:
\bele\label{lemma:Phisc}
 The functional \,$\Fhi$ is convex and lower semicontinuous in $L^1(0,T;V)$ and $D(\Fhi) \neq \emptyset$ {\rm (}i.e., $\Fhi$ is not identically $+\infty${\rm )}.
\enle
\begin{proof}
Convexity follows immediately from the definition. Moreover,  it is obvious that $D(\Fhi)$ is not empty. To show the lower semicontinuity, let $(u_n)$ and $u$ be such that $u_n \to u$ in $L^1(0,T,V)$ and $\liminf_{n \to +\infty} \Fhi(u_n) < +\infty$ (otherwise, nothing remains to prove). Then it follows from  Fatou's lemma  that
\begin{align*}
 \int^T_0 \left( \liminf_{n\to+\infty} \fhi(u_n(t)) \right)\, \d t \leq \liminf_{n \to +\infty} \Fhi(u_n(t)) \, \d t < +\infty,
\end{align*}
which in particular implies that $\liminf_{n\to+\infty} \fhi(u_n(t))$ is integrable over $(0,t)$ and finite for a.e.~$t \in (0,T)$. We may assume, up to a (not relabeled) subsequence, that $u_n(t)$ tends to $u(t)$ strongly in $V$ for almost every~$t\in (0,T)$. By virtue of the lower semicontinuity of $\fhi$ in $V$,
we then observe that
\begin{equation}\label{sc12}
  \fhi(u(t)) \le \liminf_{n\to+\infty} \fhi(u_n(t)) < +\infty
   \quext{for a.e.~}\/t\in(0,T).
\end{equation}
In particular, we find that $u(t) \in D(\fhi)$ for a.e.~$t \in (0,T)$. Moreover, we can also check the measurability of the function $t \mapsto \fhi(u(t))$ in $(0,T)$ by employing the (standard version of) Moreau-Yosida regularization for $\fhi$ and its fine properties (see, e.g.,~\cite{barbu,BCP}). Thus integrating \eqref{sc12} over $(0,T)$ and 
applying  once more  Fatou's lemma, we obtain $\fhi(u(\cdot)) \in L^1(0,T)$ and
$$
\int^T_0 \fhi(u(t)) \, \d t \leq \int^T_0 \left( \liminf_{n \to +\infty} \fhi(u_n(t)) \right) \, \d t \leq \liminf_{n \to +\infty} \int^T_0 \fhi(u_n(t)) \, \d t < +\infty.
$$
Consequently, $\Fhi$ is lower semicontinuous on $L^1(0,T;V)$.
\end{proof}
In the sequel, we shall often need to work with the 
restriction of $\Fhi$ onto $\calV$ (also denoted
by the same symbol $\Fhi$ for simplicity). 
It is then clear that $\Fhi$ is also convex,
lower semicontinuous (in $\calV$ by Lemma \ref{lem:emb1} below) and has a non-empty effective domain.
In what follows,  we set 
\begin{equation*}
A:=\deo \modu \quad \mbox{ and } \quad B:=\deo \fhi. 
\end{equation*}
From the general theory, it then follows
that $A$ and $B$ are maximal monotone, possibly multi-valued,
operators from $V$ to $2^{\Vp}$. As we shall see in the  sequel  (see Lemma \ref{lem:sing} below), $A$ will turn out to be an abstract realization of the function $\alpha$, and hence, $A$ is single-valued. In addition, we also define time-dependent analogues of the operators by putting 
$$
\calA:=\deq \Modu \quad \mbox{ and } \quad \calB:=\deq \Fhi.
$$
Analogously as before, $\calA$ and $\calB$ 
are maximal monotone operators from $\calV$ to $2^{\calVp}$. 
In the next section, we shall also rigorously prove that $\calA$ and 
$\calB$, as expected, represent time-dependent counterparts of $A$
and $B$, respectively. 

\subsection{Main result}

We are now ready to state our main result, which is concerned with existence of strong solutions to the initial-value problem for equation~\eqref{eqn}.
\bete\label{teo:main}
 Let~{\rm Assumptions~\ref{hp:alpha}} and {\rm~\ref{hp:B}} hold and also suppose that 
 \begin{align}\label{hp:u0}
   u_0 &\in D(\fhi),\\
  \label{hp:f}
   f &\in \calVp. 
 \end{align}
 Then, there exists at least one function
 $u: [0,T] \times \Omega\to \RR$ satisfying
 \begin{gather}
  u,~u_t \in \calV, \quad u \in C_w([0,T];X),\label{rego:u}\\
  t \mapsto \fhi(u(t)) \ \mbox{ is absolutely continuous on } [0,T]\no
 \end{gather}
 and solving the equation
 \begin{equation}\label{eq}
   A(u_t(\cdot,t)) + B(u(\cdot,t)) \ni f(\cdot,t)
    \ \mbox{ in }\,\Vp \ \mbox{ for a.e. } t\in (0,T)
 \end{equation}
 together with the initial condition
 \begin{equation*}
    u|_{t=0}= u_0
     \ \mbox{ in }\,V.
 \end{equation*}
 Moreover, for all $s,t\in[0,T]$, the following
 \emph{energy identity} holds\/{\rm :}
 \begin{equation}\label{energy}
   \fhi(u(\cdot,t)) - \fhi(u(\cdot,s)) + \int_s^t 
     \duav{\alpha(\cdot,u_t(\cdot,\tau)),u_t(\cdot,\tau)} \, \d \tau 
     = \int_s^t \duav{ f(\cdot,\tau),u_t(\cdot,\tau)} \, \d \tau.
 \end{equation}
\ente
 \beos 
\begin{enumerate}
 \item[(i)] It is worth observing from the very beginning that, in view of the results of the next section, inclusion~\eqref{eq} can be rewritten in several equivalent forms. Making explicit the section of $B(u(t))$, one has the equation,
\begin{equation}\label{eq:loc}
  A (u_t(t)) + \eta(t)  =  f(t)
   \quext{in }\,\Vp~~\quext{for a.e.~}\,t\in(0,T),
\end{equation}
where $\eta\in \calVp$ is such that $\eta(t)\in B(u(t))$ for 
a.e.~$t\in(0,T)$;  note in particular that the regularity $\eta\in \calVp$, though
not explicitly stated in Theorem~\ref{teo:main}, is a direct consequence of
\eqref{hp:f} and of the properties of $A$.  On the other hand, one may also write a 
``global'' formulation of the form
\begin{equation}\label{eq:glob}
  \calA(u_t) + \eta  =  f
   \quext{in }\,\calVp,
\end{equation}
with $\eta\in \calVp$ satisfying $\eta\in \calB(u)$.
In the next section, we shall see that formulations \eqref{eq:loc}
and \eqref{eq:glob} are in fact equivalent and we shall use either of them upon convenience.
 \item[(ii)] In particular, \eqref{eq:glob} implies
$$
\calA(u_t), \ \eta \in \calVp,
$$
which is the maximal regularity under \eqref{hp:f}.
\end{enumerate}
\eddos 
The proof of Theorem~\ref{teo:main} will essentially be detailed in Section~\ref{sec:proof} below. We need, 
however, to prepare a considerable amount of 
preliminary material and tools. This is the purpose of 
the following sections.

\section{Some lemmas on the Musielak-Orlicz modular}
\label{sec:tools}

In this section, we shall develop some lemmas relevant to the modular $\modu$ under (a part of) Assumption \ref{hp:alpha}. They will be used later to prove the main result and might also be of independent interest. Let us start with the following lemma, which also derives $\vep$-Young's inequality for the modulars associated with $\phj$ and $\phj^*$. It is straightforward from \eqref{phi-young} for homogeneous $\Phi$-functions (e.g., standard Lebesgue spaces); however, it is not the case for general (inhomogeneous) ones.

\begin{lem}[$\vep$-Young's inequality]\label{L:young}
In addition to {\rm (a)} of {\rm Assumption \ref{hp:alpha}}, suppose that $\phj^*$ satisfies the $\Delta_2$-condition. 
Then, for $\vep \in (0,1)$, there exists a constant $C_\vep > 0$ such that
$$
ab  \leq \vep \varphi(x,a) + C_\vep \varphi^*(x,b) \quad \mbox{ for any } \ a,b \in \R
  \ \mbox{ and a.e.~}\, x \in \Omega. 
$$
\end{lem}
\begin{proof}
Exploiting Young's inequality \eqref{phi-young}, for $\vep \in (0,1)$, we see that
$$
ab = \vep a (\vep^{-1}b) \leq \vep \left[ \phj(x,a) + \phj^*(x,\vep^{-1}b) \right]
\ \mbox{ for } \ a,b \in \R.
$$
Moreover, one can take $n_\vep \in \N$ such that $2^{n_\vep-1} < \vep^{-1} \leq 2^{n_\vep}$, and it then follows from Proposition \ref{P:nondec} and the $\Delta_2$-condition of $\phj^*$ that
$$
\phj^*(x,\vep^{-1}b) \leq \phj^*(x,2^{n_\vep} b) \leq K^{n_\vep} \phj^*(x,b).
$$
Thus we obtain the assertion with $C_\vep := \vep K^{n_\vep} > 0$.
\end{proof}

According to Lemma \ref{lemma:1x}, the $\Delta_2$-property along with the non-negativity of $\phj^*$ implies a superlinear growth of $r \mapsto \phj(x,r)$, that is, the first inequality of \eqref{jlow} with $K_0 > 2$. Based on this property, we can prove an important boundedness criterion for subsets of $V$ (of course, an analogous result holds in~$\calV$).
 In particular, such a property is the key tool that we shall use in order to deduce an \emph{a-priori} estimate
as a consequence of the energy principle satisfied by families of approximate solutions to our equation. 
\begin{lem}[Boundedness criteria]\label{lemma:1b}
 Let $\phj$ satisfy\/ {\rm (a), (i)} of {\rm (b)} and {\rm {(c)}} in \/{\rm Assumption~\ref{hp:alpha}.}
\begin{enumerate}
 \item[\rm (i)] Let $L\subset \Lpso$ be a set such that
 \begin{equation}\label{bounded}
   \modu(u) \le C \left( \| u \|_{L^\varphi(\Omega)} + 1 \right)
    \quext{for all }\, u \in L
 \end{equation}
for some $C>0$ independent of $u$.
Then $L$ is bounded in $\Lpso$. 
 \item[\rm (ii)] Let $L^* \subset L^{\varphi^*}(\Omega)$ be a set such that 
\begin{equation*}
\modud(u) \leq C \left( \|u\|_{L^{\varphi^*}(\Omega)} + 1\right) \ \mbox{ for all } \, u \in L^*
\end{equation*}
for some $C > 0$ independent of $u$. Then $L^*$ is bounded in $L^{\varphi^*}(\Omega)$.
\end{enumerate} 
\end{lem}
\begin{proof}
One can assume $C \geq 1$ without loss of generality. We first show (i). Let $u \in L$
and set $ \lambda := \| u \|_{L^\phj(\Omega)}$. Let $n\ge 1$ be such that $\lambda\in (2^{n-1},2^n]$
(if $\lambda \le 1$, there is nothing to prove). Recalling Lemma \ref{lemma:1x}, we have
\begin{align*}
 C \left(1 + \frac 1 \lam\right) &\ge \io \frac{\phj(\cdot,|u|)}{\lambda} \, \d x\\
&\ge \io \frac{\phj\left(\cdot,2^{n-1} |u| 2^{1-n}\right)}{2^n} \, \d x
   \stackrel{\eqref{jlow}}\ge \frac{K_0^{n-1}}{2^n} \io \phj\left(\cdot,\frac{|u|}{2^{n-1}} \right) \, \d x 
   \ge \frac{K_0^{n-1}}{2^n}. 
\end{align*}
Here, the last inequality follows from the definition of the (Luxemburg-type) norm along with $2^{n-1} < \lambda = \| u \|_{L^\phj(\Omega)}$. Then, we obtain
\begin{equation*}
  \left( \frac{K_0}{2} \right)^n 
   \le 2 C K_0,
\end{equation*}
which along with the fact that $K_0 > 2$ implies
\begin{equation*}
  n \le \log_2 ( 2 C K_0 ) / \log_2 ( K_0/2 )
\end{equation*}
and, consequently, for any $u \in  L$, one has either $\| u \|_{L^\phj(\Omega)}\le 1$ or
\begin{equation}\label{low14}
\| u \|_{L^\phj(\Omega)} = \lambda \le 2^{ \log_2 ( 2 C K_0 ) / \log_2 ( K_0/2 ) } = ( 2 C K_0 )^{ 1 / \log_2 (K_0/2)}, 
\end{equation}
as desired. It is worth noting that, as expected, the closer is $K_0$ to $2$, the larger is the bound on the \rhs\ of \eqref{low14}. The assertion (ii) can also be verified in the same manner.
\end{proof}

We also remark that \eqref{bounded} can be restated as follows: under Assumption \ref{hp:alpha},
the modular functional $\modu$ is {\sl coercive}\/ with respect to the corresponding Luxemburg-type norm, namely one has
\begin{equation}\label{bounded2}
  \lim_{\| u \|_{L^\phj(\Omega)} \to +\infty}\frac{\modu(u)}{\| u \|_{L^\phj(\Omega)}} = +\infty.
\end{equation}
\beos\label{rem:bound}
The above result may fail to be true when \eqref{jlow} does not hold.
To see this, let us consider the case when  all  the elements of $L$ are 
of the form $u = k \chi_E$ for $k > 0$ and a measurable subset $E$ of $\Omega$  with $|E|>0$, and $\phj$ is independent of $x$, i.e., $\varphi(x,r)=\varphi(r)$, which will be specified later.
Then, by continuity of $\phj(\cdot)$,  we have
\begin{equation}\label{bo1}
  \lambda = \| \chi_E \|_V \ \Leftrightarrow \ 
   \iE \phj\left(\frac1\lambda\right) \, \d x= 1 \ \Leftrightarrow \ 
   \phj\left(\frac1\lambda\right) = \frac1{|E|}.
\end{equation}
 Moreover, assuming that, for all $u\in L$, $\| u \|_V$ is large enough 
(otherwise there is nothing to prove), we can get rid of the summand $1$ in brackets
in \eqref{bounded}, which can then be  rewritten as
\begin{equation}\label{bo2}
  \phj(k)|E| \le C \| k \chi_E \|_V 
   = C k \| \chi_E \|_V.
\end{equation}
Combining \eqref{bo1} with \eqref{bo2} we obtain
\begin{equation}\label{bo3}
  \frac{\phj(k)}k \le C \frac{\phj\left(\frac1\lambda\right)}{\frac1\lambda}.
\end{equation}
Put now $\ell:= \lambda^{-1}$. 
Set also, for $r\in (0,\infty)$, $m(r):= \phj(r)/ r$. A direct check
shows that $m(r)$ is not decreasing; indeed, we see that $m'(r) = (\varphi'(r)r - \varphi(r))/r^2 = \varphi^*(\varphi'(r))/r^2 \geq 0$.
Moreover, from  Lemma~\ref{L:Nfun}, \eqref{Nfun} holds, i.e., 
\begin{equation}\label{bo4}
  \lim_{r\searrow 0} m(r) = 0, \quad
  \lim_{r\nearrow +\infty} m(r) = +\infty.
\end{equation}
Hence, combining \eqref{bo3} with \eqref{bo4} we infer that
\begin{equation}\label{bo3c}
  \frac{m(k)}{m(\ell)} \le C.
\end{equation}
Actually, a bound for $u = k \chi_E$ would follow from \eqref{bo3c}
if it were possible to deduce from it that $k\lambda=k/\ell$ is bounded.
However, this is readily seen to be false, at least in general.
Indeed, it is sufficient
to take $\phj(r) = r \log^2(1+r)$ (note that $\phj$
is a generalized uniformly convex $N$-function; however, $\phj^*$
has an exponential growth at infinity and, hence, does
not satisfy the $\Delta_2$-condition), corresponding
to $m(r) = \log^2 (1 + r)$. Then, we take 
$$
  u_n = k_n \chi_{E_n}, \quad
  k_n = n \log (1 + n), \quad
   \ell_n = n,
$$
corresponding to
$$
  \lambda_n = \frac1n, \quad
   | E_n | = \frac1{n \log^2(1+n)}.
$$
Then, it is clear that the quotient $k_n / \ell_n$ diverges,
whereas $m(k_n) \le c m(\ell_n)$ for all $n$.

Moreover, we can observe that,  for large $r>0$, 
$$
  \alpha(r)=\phj'(r)=\log^2(1+r)+\frac{2r}{1+r}\log(1+r)
   \sim\log^2(1+r).
$$
Furthermore, we can easily compute 
\begin{equation}\label{un-div}
  \| u_n \|_{L^\phj(\Omega)} = k_n \lambda_n  = \log( 1 + n)\nearrow \infty.
\end{equation}
 On the other hand, we have 
\begin{align*}
  \frac1{\| u_n \|_{L^\phj(\Omega)}} \io \phj(u_n) \, \d x
  & = \frac1{\log ( 1 + n ) } \int_{E_n} \phj ( n \log( 1 + n )) \, \d x\\
  & \sim \frac1{\log ( 1 + n )} | E_n| n \log ( 1 + n ) \log^2 \big( 1 + n \log (1+n) \big)\\
  & \sim \frac1{\log ( 1 + n )} \frac1{n \log^2(1+n)}  n \log^3 ( 1 + n )
   \sim 1,
\end{align*}
so that the quotient on the \lhs\ is bounded uniformly in~$n$. Comparing
with \eqref{un-div}, we then see that \eqref{bounded2} fails. \qed 
\eddos

We close this subsection with the following three lemmas.

\begin{lem}\label{L:snd}
Let {\rm (a)} and {\rm {(c)}} of {\rm Assumption \ref{hp:alpha}} hold. 
Then, for any $\lam > 0$, there exists a constant $C_\lam>0$ depending on $\lam$ such that
\begin{equation}\label{equiv-phi}
C_\lam^{-1} \varphi(x, r) \leq \varphi(x,\lam r) \leq C_\lam \varphi(x, r)
\ \mbox{ for a.e. } x \in \Omega \ \mbox{ and } \ r \geq 0.
\end{equation}
Moreover, the assertion above is also true for the conjugate $\phj^*$.
\end{lem}

\begin{proof}
In case $\lambda > 1$, one can take $n \in \N$ such that $2^{n-1} < \lam \leq 2^n$. Then we see that
$$
\varphi(x,2^{n-1}r) \leq \varphi(x,\lam r) \leq \varphi(x,2^nr) \quad \mbox{ for } \ r \geq 0.
$$
Using {(c)} of Assumption \ref{hp:alpha} and Lemma \ref{lemma:1x}, we observe that
$$
\varphi(x,2^{n-1}r) \geq K_0^{n-1} \varphi(x,r)
\quad \mbox{ and } \quad
\varphi(x,2^n r) \leq K^n \varphi(x,r).
$$
Thus \eqref{equiv-phi} follows. In case $0 < \lam \leq 1$, let $n \in \N$ be such that 
$2^{-n} < \lam \leq 2^{-n+1}$. The rest of proof runs as in the other case. Moreover, 
one can prove the same assertion for $\phj^*$ in the same way. 
\end{proof}

\begin{lem}\label{L:equiv}
In addition to {\rm (a)} and {\rm (i)} of {\rm (b)} in {\rm Assumption \ref{hp:alpha}}, 
suppose that $\phj$ satisfies the $\Delta_2$-condition, which is a part of {\rm {(c)}} 
of {\rm Assumption \ref{hp:alpha}}. Let $(u_n)$ be a sequence in $L^\phj(\Omega)$. 
Then, the following conditions are equivalent to each other\/{\rm :}
\begin{enumerate}
 \item[(i)] $u_n \to 0$ strongly in $L^\phj(\Omega)$\/{\rm ;}
 \item[(ii)] $\modu(u_n) \to 0$.
\end{enumerate}
Moreover, the equivalence above still holds true  whenever one considers
a convergence property that is \emph{uniform} with respect to some additional 
parameter. 
\end{lem}

The above lemma is proved in~\cite[Lemma 2.1.11]{DHHR}, where equivalence for uniform convergence is however not mentioned. We only need a slight modification to fill the gap; however, for the completeness, we give a proof.

\begin{proof}
Let $(u_n^\ell)$ be a sequence in $L^\phj(\Omega)$ depending on a parameter $\ell$.
We first assume that $u_n^\ell \to 0$ strongly in $V$ and uniformly in $\ell$, as $n \to +\infty$. We use the fact that
$$
\rho(x) \leq \|x\|_\rho \quad \mbox{ if } \ x \in X_\rho, \ \|x\|_\rho \leq 1
$$
for general semimodular space $X_\rho$ (see~\cite[Corollary 2.1.15]{DHHR}). Hence one immediately obtains $\modu(u_n^\ell) \to 0$ uniformly in $\ell$ as $n \to +\infty$. We next suppose that $\modu(u_n^\ell) \to 0$ uniformly in $\ell$ as $n \to +\infty$. We claim that, for any $\lam > 1$, $\modu(\lam u_n^\ell) \to 0$ uniformly in $\ell$ as $n \to +\infty$. Indeed, one can take $m \in \N$ such that $2^{m-1} < \lam \leq 2^m$. Hence it follows from \eqref{K2} that $\modu(\lam u_n^\ell) \leq \modu(2^m u_n^\ell) \leq K^m \modu(u_n^\ell) \to 0$ uniformly in $\ell$ as $n \to +\infty$. We further deduce that $\sup_{\ell} \|u_n^\ell\|_{L^\phj(\Omega)} \leq 1/\lam$ for $n$ large enough. Thus $u_n^\ell \to 0$ strongly in $V$ uniformly in $\ell$ as $n \to +\infty$. In particular, in the case when $u_n^\ell = u_n$ (i.e., it is independent of $\ell$), we immediately obtain the original assertion.
\end{proof}

\begin{lem}\label{L:bdd-equiv}
In addition to {\rm (a)} and {\rm (i)} of {\rm (b)} in {\rm Assumption \ref{hp:alpha}}, 
assume that $\varphi$ fulfills the $\Delta_2$-condition. Let $(u_n)$ be 
a sequence in $L^\phj(\Omega)$. Then the following {\rm (i)} and {\rm (ii)} are equivalent\/{\rm :}
\begin{enumerate}
 \item[(i)] $(u_n)$ is bounded in $L^\phj(\Omega)$\/{\rm ;}
 \item[(ii)] $\modu(u_n)$ is bounded.
\end{enumerate}
\end{lem}

\begin{proof}
The implication (ii) $\Rightarrow$ (i) holds generally\/{\rm ;} indeed, it holds that
\begin{equation}\label{norm<modu}
\|u\|_\rho \leq \rho(u)+1 \ \mbox{ for } u \in X_\rho
\end{equation}
for any semimodular space $(X_\rho, \|\cdot\|_\rho)$ (see~\cite[(c) of Lemma 2.1.15]{DHHR}). Hence it suffices to prove the inverse implication. Let $(u_n)$ be a bounded sequence in $L^\phj(\Omega)$, say $\|u_n\|_{L^\phj(\Omega)} \leq C$ for $n \in \N$. We can assume $C \geq 1$ without loss of generality. Then it follows from the definition of $\|\cdot\|_{L^\phj(\Omega)}$ that
$$
\modu (u_n/C) \leq 1.
$$
On the other hand, one can take $k \in \N$ (independent of $n$) such that $2^{-k} < 1/C \leq 2^{-k+1}$, and hence,
\begin{equation}\label{1}
\modu(u_n/2^{k}) \leq 1
\end{equation}
(see Proposition \ref{P:nondec}). Employing the $\Delta_2$-condition of $\varphi$ (see \eqref{K2}), we deduce that
\begin{align*}
\varphi(x,u_n(x)) \leq K^{k} \varphi(x,u_n(x)/2^{k}) \ \mbox{ for a.e. } x \in \Omega,
\end{align*}
which along with \eqref{1} implies
$$
\modu(u_n) \leq K^{k} < +\infty\ \mbox{ for } n \in \N.
$$
Thus the proof is completed.
\end{proof}

\section{Subdifferentials in Musielak-Orlicz spaces}
\label{subsec:convex}

In this section, we shall develop a number of tools, related to
subdifferentials and duality methods in Musielak-Orlicz spaces.
This machinery will play a key role in the proof of Theorem \ref{teo:main} in \S \ref{sec:proof}.

\subsection{Characterization of the operator $A$}

Our first purpose stands in characterizing a bit more the operator $A = \partial_\Omega \modu$. We start with
\bele\label{lemma:1}
 Let \/{\rm Assumption \ref{hp:alpha}} hold and let $u\in \Lpso$. Let $\xi : \Omega \to \RR$ be a function defined by $\xi(x) = \alpha(x,u(x))$ for a.e.~$x \in \Omega$.
 Then $\xi \in \Lpsostar$ and $\xi u \in L^1(\Omega)$. 
 Moreover, it holds that $\xi \in A(u) = \deo \modu(u)$.
\enle
\begin{proof}
Let us recall that, under the present assumptions, $\phj$ is differentiable in~the second
variable and $\alpha(x,r) = \de \phj(x,r) = \frac{\partial \phj}{\partial r}(x,r)$
for all $r\in\RR$ and a.e.~$x\in \Omega$. Namely, the subdifferential coincides
with the partial derivative with respect to $r$. Moreover, 
without loss of generality, we may assume $u \in L^{\phj}(\Omega)$ to be nonnegative. Indeed, we find by $\alpha(x,0)=0$ that $\alpha(x,u(x)) = \alpha(x,u_+(x))-\alpha(x,u_-(x))$, where $u_\pm(x) := \max \{\pm u(x),0\}$.
Then, for $n\in \NN$, we define the truncated function 
$u_n:=\min\{n,u\}$ and we also set $\xi_n(x)=\alpha(x,u_n(x))$. 
Then, it is easy to check that both sequences $(u_n)$ 
and $(\xi_n)$ are increasing. Moreover,
$\xi_n$ is measurable in $\Omega$, since $\alpha$ is Carath\'eodory 
and $u_n$ is measurable in $\Omega$. Then, by the definition of 
subdifferential, we have
\begin{align}
0 &\leq  \xi_n(x) u_n(x) \nonumber \\
 &= \xi_n(x) \big( 2 u_n(x) - u_n(x) \big)
    \le \phj(x,2 u_n(x)) - \phj(x, u_n(x)) 
\label{co11}
\end{align}
for a.e.~$x \in \Omega$.
Here we used $\alpha(x,0)=0$, i.e., $\alpha(x,r)r \geq 0$ for $r \in \R$.
Hence, invoking {(c)} of Assumption~\ref{hp:alpha}, we infer that
\begin{align}
0 \leq \xi_n(x) u_n(x) 
 &\le (K-1) \phj(x,u_n(x)) \nonumber \\
 &= (K-1) \phj(x,|u_n(x)|) 
   \le (K-1) \phj(x,|u(x)|)\label{co12}
\end{align}
for a.e.~$x \in \Omega$. Here we also used the fact that $r \mapsto \varphi(x,r)$ 
is even and increasing on $[0,+\infty)$ for a.e.~$x \in \Omega$. Consequently, 
noting that $u_n \to u$ and $\xi_n\to \xi = \alpha(\cdot,u(\cdot))$ (by the 
continuity of $\alpha(x,r)$ in $r$) a.e.~in $\Omega$ and applying Lebesgue's 
dominated convergence theorem, we deduce that
\begin{equation}\label{co13}
  \xi u \in L^1(\Omega) 
   \quext{and } \io \xi u \, \d x \le (K-1) \io \phj(\cdot,|u|) \, \d x
    = (K-1) \modu(u). 
\end{equation}
Now, by the use of the Fenchel-Moreau identity, we see that
\begin{equation}\label{co14}
  \phj^*(x,\xi_n(x)) = \xi_n(x) u_n(x) - \phj(x,u_n(x))
   \le \xi_n(x) u_n(x) 
   \le \xi(x) u(x) 
\end{equation}
for a.e.~$x\in \Omega$.
Hence, on account of \eqref{co13}, we can apply once more the dominated convergence theorem to get
\begin{equation*}
  \io \phj^*(\cdot,|\xi|) \, \d x
   = \lim_{n\to+\infty} \io \phj^*(\cdot,|\xi_n|) \, \d x 
   = \io \big( \xi u - \phj(\cdot,|u|) \big) \, \d x
   \in \RR.
\end{equation*}
 This implies in particular that 
\begin{equation*}
  \modud(\xi) = \io \phj^*(\cdot,|\xi|) \, \d x
   = \duav{\xi,u} -\modu(u) < +\infty. 
\end{equation*}
Consequently, we have obtained $\xi \in L^{\phj^*}(\Omega)$ and 
$\xi \in \deo\modu(u)$ from the fact that $\modud = (\modu)^*$
(see (vii) of Proposition \ref{P:proper}), as desired. 
\end{proof}
We have essentially proved that, whenever $u\in L^\varphi(\Omega)$, then the ``pointwise''
function $\alpha(x,u(x))$ is an element of the set $A(u)$,
which may, in principle, contain more than one element of $\Vp$.
However,  under our assumptions, this can in fact never occur, 
because  $\alpha$ is monotone, continuous and coercive.
\begin{lem}[Representations of $A$ and $\calA$]\label{lem:sing}
 Let\/ {\rm Assumption~\ref{hp:alpha}} hold and let $u\in V$.
 Then $A(u) = \{\alpha(\cdot,u(\cdot))\}$. Moreover, for $u \in \calV$, 
 it holds that $\calA(u) = \{\alpha(\cdot,u(\cdot,\cdot))\} = A(u(\cdot))$.
\end{lem}
\begin{proof}
Let $u \in V$ and suppose on the contrary that $\deo \modu(u)$ 
contains an element $\xi$ which differs from $\alpha(\cdot,u)$
on a subset of $\Omega$ having strictly positive measure.
In particular, we may assume that there exist $\epsi>0$
and a measurable set $E\subset \Omega$ with $|E|>0$ such that
$\xi(x) \ge \alpha(x,u(x)) + \epsi$ for $x \in E$. Let us then set
\begin{equation*}
  v(x):= \begin{cases}
    u(x) &\text{if }\, x\in  \Omega\setminus E ,\\
    u(x) + \delta(x) &\text{if }\, x\in  E ,
   \end{cases} 
\end{equation*}
where $\delta(x)>0$ is chosen in such a way that
$\alpha(x,u(x)+\delta(x)) = \alpha(x,u(x)) + \epsi/2$
due to the continuity and coercivity of $\alpha(x,r)$ 
in $r$ (indeed, the latter follows from \eqref{Nfun}).
In view of the fact that $\alpha(\cdot,v) \in A(v)$ due to the 
previous lemma, we then observe that
\begin{equation*}
  0 \le \duav{ \xi - \alpha(\cdot,v), u - v } 
   = \int_E ( \xi - \alpha(\cdot,v) ) ( - \delta ) \, \d x
   \le - \frac\epsi2 \int_E \delta \, \d x,
\end{equation*}
whence $\epsi = 0$, that is, a contradiction. Furthermore, repeating the argument so far
(including Lemma \ref{lemma:1}) with $\Omega$ replaced by $Q = \Omega\times(0,T)$, 
one can verify that, for $u \in \calV$, the set $\calA(u)$ consists of the function 
$(x,t) \mapsto \alpha(x,u(x,t))$ only.
\end{proof}

Thanks to the previous lemmas, the operator $A$ in equation \eqref{eq} can be 
interpreted ``pointwisely'' in such a way that \eqref{eq} can 
be handled in the ``concrete'' form \eqref{eqn}.

\subsection{Maximality criteria based on the Musielak-Orlicz modular}\label{Ss:max}

In view of the fact that the operator $A$ is tied to the choice of the space $V$, we shall see that $A$ 
enjoys further important properties. The following result, extending \cite[Theorem 1.2, Chap.~II, p.~39]{barbu}, 
tells us that the operator $A$ can be used, in place of the duality mapping of $V$, in order to 
characterize the maximality of nonlinear monotone operators from $V$ to $2^{\Vp}$. This property 
will play a basic role in our method for approximating equation~\eqref{eqn} (see \S \ref{sec:proof})
as well as in proving a chain-rule formula customized for the 
Orlicz-Musielak setting (see \S \ref{subsec:chain}).
\begin{thm}[Maximality criteria based on the modular]\label{thm:mamo}
 Let {\rm Assumption \ref{hp:alpha}} hold. Let $S : V \to 2^{V^*}$ be a
 {\rm (}possibly nonlinear\/{\rm )} monotone operator. 
 Then, $S$ is maximal if and only if for every, or some,
 $\lambda>0$, $S+\lambda A$ is surjective with $A = \partial_\Omega \modu$.
\end{thm}
\begin{proof}
Assume that $S + \lambda A$ is surjective. Then, the maximality of $S$ may be proved by following the lines of \cite{barbu} (see also~\cite{BCP}). Indeed, let $[u_0,\xi_0] \in V \times V^*$ be such that
 \begin{equation}\label{SlA}
  \langle \xi_0-\eta, u_0-v \rangle_V 
  \geq 0 \quad \mbox{ for all } \ [v,\eta] \in G(S),
 \end{equation}
 where $G(S)\subset V\times V^*$ denotes the graph of $S$. Then, due to the surjectivity of $S + \lambda A$, one can take $u_1 \in D(S)$ and $\xi_1 \in S(u_1)$ such that
 \begin{equation}\label{xiA}
 \xi_1 + \lambda A(u_1) = \xi_0 + \lambda A(u_0) \in \Vp.
 \end{equation}
 Substituting $v = u_1$ and $\eta = \xi_1$ to \eqref{SlA}, one has
 $$
  \langle \xi_0-\xi_1,u_0-u_1 \rangle_V \geq 0, 
 $$
 which along with \eqref{xiA} implies
 $$ 
   \lambda \langle A(u_1)-A(u_0),u_0-u_1 \rangle_V 
   \geq 0.
 $$
 Since $A$ is strictly monotone (by the strict convexity of $\phj(x,\cdot)$, 
 see (a) of Assumption \ref{hp:alpha}), we deduce that $u_0=u_1$, and hence, 
 $\xi_0 = \xi_1$ by \eqref{xiA}. Thus we obtain $[u_0,\xi_0] = [u_1,\xi_1] \in G(S)$. 
 Therefore, $S$ turns out to be maximal. 

Conversely, assume that $S$ is maximal. To prove the surjectivity of $S+\lam A$ for any $\lam>0$, we shall use the demicontinuity, boundedness and coercivity of $A:V\to V^*$, which will be proved just below. With these properties of $A$, we can apply~\cite[Theorem 1.1, Chap.\,II, p.\,34]{barbu} and conclude the proof following the lines of~\cite[Proof of Theorem 1.2, Chap.\,II, p.\,39]{barbu}. The details are left to the reader.
\end{proof}

The properties of $A$ used above will be proved in the following

\begin{lem}\label{L:A-conti}
Under \/{\rm Assumption \ref{hp:alpha}}, the operator $A:V \to V^*$ is demicontinuous, bounded and coercive, and moreover, so does $A^{-1}:V^* \to V$. In addition, if $\modu$ is uniformly convex, then $A : V \to V^*$ is continuous.
\end{lem}

\begin{proof}
We recall that $A$ is a monotone mapping from $V$ to $\Vp$. We first show that $A$ is  demicontinuous. 
To this end, let $(u_n)$ be a sequence in $V$ such that $u_n\to u$ strongly in $V$ for some $u \in V$. By virtue of (vi) of Proposition \ref{P:proper}, this implies strong convergence in $L^1$ and, up to a (not relabeled) subsequence, pointwise convergence. Fix $x \in \Omega$ so that $u_n(x) \to u(x)$ and $r \mapsto \alpha(x,r)$ is continuous. Letting $\xi_n = A(u_n)$, we have $\xi_n(x) \to \xi(x) := \alpha(x,u(x))$ by the continuity of $\alpha(x,\cdot)$. Moreover, proceeding similarly to \eqref{co11}--\eqref{co12}, we may deduce that
\begin{equation*}
0 \leq \xi_n(x) u_n(x) \le \phj(x,2 u_n(x)) - \phj(x, u_n(x))
   \le (K-1)  \phj(x,|u_n(x)|) 
\end{equation*}
for a.e.~$x \in \Omega$.
Combining this fact with the analogue of \eqref{co14},
we obtain
\begin{equation*}
  \phj^*(x,|\xi_n(x)|)
   \le (K-1) \phj(x,|u_n(x)|) \
   \mbox{ for a.e. } x \in \Omega.
\end{equation*}
Notice that, after integration, the \rhs\ is bounded by Lemma \ref{L:bdd-equiv} 
along with the boundedness of $(u_n)$ in $V$. Consequently, thanks to Lemma \ref{L:bdd-equiv} again, $(\xi_n)$ turns out to be bounded in $\Lpsostar$.
By reflexivity, there exists a subsequence $(n_k)$ of $(n)$ such that $\xi_{n_k} \to \xi$ weakly in $\Vp$. 
Here we used the coincidence of pointwise and weak limits in $L^1(\Omega)$. Moreover, due to the uniqueness of the limit, we can obtain the convergence of the whole sequence $(\xi_n)$. Hence, $A$ is demicontinuous. 

As for the case where $\modu$ is uniformly convex, since $u_n \in \partial (\modu)^*(\xi_n) = \partial \modud(\xi_n)$,
it follows that
$$
\modud(\xi_n) \leq \modud(\xi) + \langle \xi_n - \xi , u_n \rangle.
$$
 Taking the supremum limit, noting that the right-hand side converges to $\modud(\xi)$, and recalling the (weak) lower semicontinuity of $\modud$ in $V^*$, we then deduce 
$$
\modud(\xi_n) \to \modud(\xi).
$$
Thanks to~\cite[Lemma 2.4.17]{DHHR} as well as Lemma \ref{L:equiv}, we deduce that $\xi_n\to\xi$ strongly in $V^*$. Therefore $A:V\to V^*$ is continuous.

We next prove that $A:V\to V^*$ is bounded. Let $L$ be a bounded set in $L^\phj(\Omega)$ and note that
$$
\langle A(u),u \rangle = \modu(u) + \modud(A(u)) \geq \modud(A(u))
\ \mbox{ for } \ u \in L,
$$
whence follows that there exists $C > 0$ such that
$$
\modud(A(u)) \leq C \|A(u)\|_{L^{\varphi^*}(\Omega)} \ \mbox{ for } \ u \in L,
$$
which along with Lemma~\ref{lemma:1b} implies the boundedness 
of $A(L)$ in $L^{\varphi^*}(\Omega)$. Hence $A : V \to V^*$ is a bounded operator.

Finally, let us show that $A$ is coercive. This corresponds to proving that
\begin{equation*}
   \frac1{\| u \|_V} \duav{A(u), u} 
    = +\infty
    \quad \mbox{ as } \ \| u \|_V \to +\infty. 
\end{equation*}
Note that
\begin{align}\no
  \frac1{\| u \|_V} \duav{A(u), u} 
   & = \frac1{\| u \|_{V}}
           \left( 
\modu(u) + \modud(A(u))
\right)
 \label{co24}
\ge \frac{\modu(u)}{\| u \|_V}.
\end{align}
Then the coercivity of $A : V \to V^*$ follows from Lemma~\ref{lemma:1b}. Indeed, if the left-hand side of the above is bounded for $u$ lying on a set $L$, then the set $L$ turns out to be bounded in $V$ due to Lemma \ref{lemma:1b}. Hence $A$ is coercive. 

Since $\varphi^*$ fulfills all the conditions as in Assumption \ref{hp:alpha} 
for $\varphi$ (see Lemma \ref{L:phj*}), one can also assure the demicontinuity, boundedness and coercivity of $A^{-1} = \partial_\Omega \modud : V^*\to V$. 
\end{proof}
\beos\label{mamo:vec}
With obvious modifications, one can easily prove a vector-valued analogue of Theorem~\ref{thm:mamo} and Lemma \ref{L:A-conti}. Namely, if $\calS$ is a monotone operator from $\mathcal V$ to $2^{\mathcal V^*}$, then the maximality of~$\calS$ is equivalent to the surjectivity of $\calS + \lambda \calA$ for every, or some, $\lambda>0$. Moreover, $\calA:\calV \to \calV^*$ is  demicontinuous,  bounded and coercive.
\eddos

Let us give a further auxiliary lemma, which will be used to reveal the relation between the operators $\calB$ and $B$ in the next section.
\begin{lem}\label{lemma:3a}
 Let~{\rm Assumptions~\ref{hp:alpha}} hold and let $S:V \to 2^{V^*}$ be maximal monotone. Set $T:= (A + S)^{-1}$ as 
 an operator from $\Vp$ to $V$. Then $T$ is 
 demicontinuous.
\end{lem}
\begin{proof}
First of all, let us observe that $T$ is well defined. Indeed, $A+S$ is bijective due to Theorem \ref{thm:mamo} 
(here the injectivity of $A + S$ follows from the strict monotonicity of $A$).
Let then $(g_n)$ be a sequence in $V^*$ such that $g_n \to g$ strongly in~$\Vp$ 
for some $g \in V^*$ and set $w_n:= T(g_n)$ and $w:= T(g)$. This corresponds to the relations,
\begin{equation}\label{T11}
  g_n = A(w_n) + b_n, \quad
  g = A(w) + b, \quad 
  b_n \in S(w_n), \quad 
  b \in S(w).
\end{equation}
Then, we shall show that $w_n$ converges to $w$
weakly in $V$. To this aim, subtract $b_* \in S(w_*)$ for some $w_* \in D(S)$ from the first equation in \eqref{T11} and test both sides by $w_n-w_*$. Then it easily follows that
\begin{align}
 \modu(w_n) -\modu(w_*)
 &\le \duav{A(w_n),w_n-w_*} + \duav{b_n-b_*,w_n-w_*}\no\\
 &= \duav{g_n-b_*,w_n-w_*}\no\\
 &\le \left( \|g_n\|_{\Vp} + \|b_*\|_{\Vp} \right) \left( \| w_n \|_V + \|w_*\|_{\Vp} \right)\no\\
 &\le c \left( \| w_n \|_V + \|w_*\|_{\Vp} \right).\no
\end{align}
Hence, $(w_n)$ is bounded in $V$ thanks to Lemma~\ref{lemma:1b}.
By the reflexivity of $V$, $(w_n)$ admits a (not relabeled) weakly convergent
subsequence.

Let us now take the difference of the two equations in \eqref{T11}
and test it by $w_n - w$. Using the monotonicity of $S$ along with the strong convergence $g_n \to g$ in~$\Vp$,
we infer that
\begin{equation*}
  \io \big( \alpha(\cdot,w_n) - \alpha(\cdot,w) \big) (w_n-w) \, \d x
   \le  \| g_n - g \|_{V^*} \big( \| w_n \|_V + \| w \|_V \big)
   \to 0,
\end{equation*}
whence the nonnegative function 
$m_n := ( \alpha(\cdot,w_n) - \alpha(\cdot,w) ) (w_n-w) $
tends to $0$ in $L^1(\Omega)$ and, up to a subsequence,
almost everywhere. Thanks to the strict convexity
of $\phj(x,\cdot)$ for a.e.~$x \in \Omega$, this implies that 
$w_n \to w$ almost everywhere in $\Omega$. Combining this property
with the boundedness of $(w_n)$ in $V$, we can verify that
$w_n\to w$ 
weakly in $V$. 
By virtue of the uniqueness of the limit, one can easily prove that the weak convergence holds for the whole sequence $(w_n)$.
\end{proof}

\subsection{Characterization of the operator $\calB$}

This subsection provides a characterization of the operator $\calB$ as a representation of $B : V \to 2^{V^*}$ in the (space-time) Musielak-Orlicz space $\calV = L^\phj(Q)$. To be more precise, we shall prove that 
\begin{center}
$b \in \calB(u)$ is equivalent to $b(t) \in B(u(t))$ for a.e.~$t \in (0,T)$  
\end{center}
for $u \in \calV$ and $b \in \calV^*$.

\begin{lem}[Relation between $\calB$ and $B$]\label{lemma:3}
 Let {\rm Assumptions~\ref{hp:alpha}} and {\rm \ref{hp:B}} hold. 
 Let $\Bext : \calV \to \calVp$ be given by, for $u \in \mathcal V$ and $\xi \in \calV^*$,
 \begin{equation}\label{Bext}
   \xi \in \Bext(u) 
\ \stackrel{\text{{\rm define}}}{\Longleftrightarrow} \ 
    \xi(\cdot,t) \in B(u(\cdot,t))
    \quext{for a.e.~}\,t\in(0,T),
 \end{equation}
 where the domain of $\Bext$ is given by
 the set of functions $u\in \calV$ satisfying that there
 exists {\rm (}at least\,{\rm )} one $\xi\in \calVp$ for which\/ 
 \eqref{Bext} holds. 
 Then the operators $\calB=\partial_Q\Fhi$ and $\Bext$ do coincide.
\end{lem}
\begin{proof}
First of all, due to the definition of subdifferential, 
for any $u\in D(\Bext)$, $\xi \in \Bext(u)$ and $v\in \calV$,
the following holds:
\begin{equation*}
  \duav{ \xi(t) , v(t) - u(t) }  
    + \fhi(u(t))
   \le \fhi(v(t))   
   \quext{for a.e.~}\,t\in(0,T).
\end{equation*}
Then, integrating it over $(0,T)$, one readily obtains $\Bext \subset \calB$.

To get the converse inclusion, which is 
more delicate, we shall prove that $\Bext$ is
a maximal monotone operator from $\calV$ to $2^{\calVp}$.
To this aim, in view of Remark \ref{mamo:vec} along with Lemma \ref{lem:sing},
it is sufficient to show that, for every
$g\in \calVp$, there exists $u\in D(\Bext)$
such that
\begin{equation}\label{Bext11}
  A(u(t)) + B(u(t)) \ni g(t) \ \mbox{ in } V^* 
  \quext{for a.e.~}\,t\in(0,T).
\end{equation}
Noting that $\calV^* \subset L^1(0,T;V^*)$ (see Lemma \ref{lem:emb1} below), one can take a sequence $(g_n)$ 
in $C^0([0,T];V^*)$ such that $g_n \to g$ strongly in $L^1(0,T;V^*)$ and $g_n(t)\to g(t)$ strongly in $V^*$ for a.e.~$t \in (0,T)$. Set $T := (A+B)^{-1}$ (see Lemma \ref{lemma:3a}). Then, for {\sl every}\/ $t\in[0,T]$, since $g_n(t)$ lies in $\Vp$, we can define $u_n(t) := T(g_n(t)) \in D(B)$ as the unique function satisfying
\begin{equation*}
A(u_n(t)) + b_n(t) = g_n(t),
\quad b_n(t) \in B(u_n(t)) \ \mbox{ in } V^*.
\end{equation*}
Thanks to the demicontinuity of $T$ proved in Lemma~\ref{lemma:3a}, we then deduce that 
$u_n \in C_w([0,T];V)\subset \calV$ for all $n\in\NN$. In particular, $u_n : (0,T) \to V$
is strongly measurable. Moreover, recalling $g_n(t)\to g(t)$ strongly in $V^*$ and using 
Lemma \ref{lemma:3a}, we infer that $u_n(t) \to u(t) := T(g(t))$ weakly in $V$ as $n \to +\infty$. 
Hence by Pettis' theorem, $u:(0,T) \to V$ turns out to be strongly measurable. Moreover, $u(t)$ satisfies \eqref{Bext11}. It remains to check $u \in \calV$.
Let $u_* \in D(B)$ and subtract $b_* \in B(u_*)$ from both sides of \eqref{Bext11}. 
Test it by $u(t)-u_*$ and apply Young's inequality (see Lemma \ref{L:young}). 
With the aid of the monotonicity of $B$, it then follows that
\begin{align}\label{int11}
\modu(u(t)) &\leq \modu(u_*) + \langle g(t) - b_* , u(t) - u_* \rangle\\
&\leq \frac 1 2 \modu(u(t)) + C \left( \modud(g(t))  + \modud(b_*)  + \modu(u_*) \right) + \modu(u_*),
\end{align}
which along with $g \in \calV^*$ (i.e., $\modud(g(\cdot)) \in L^1(0,T)$) implies $\modu(u(\cdot)) \in L^1(0,T)$.
Furthermore, recalling \eqref{norm<modu}, one deduces that $u \in L^1(0,T;V) \subset L^1(0,T;L^1(\Omega))$, 
which also ensures that $u = u(x,t)$ is (Lebesgue) measurable in $\Omega\times(0,T)$ 
(see~\cite[Proposition A.3]{AS} for more details). Therefore integrating \eqref{int11} over $(0,T)$ and using Fubini's lemma, we obtain $u \in \calV$ from the fact that $\modu(u(\cdot)) \in L^1(0,T)$.
\end{proof}
\beos\label{nomeas}
 Apparently, a simpler proof of \eqref{Bext11} could be provided just by fixing 
 $t\in(0,T)$ and noting that, thanks to the maximality of $B$ as an operator from $V$ to $2^{\Vp}$, by Theorem~\ref{thm:mamo}, there exists a function $u(t)$ satisfying \eqref{Bext11}. 
 However, in such a way we may not guarantee that $t \mapsto u(t)$ is strongly 
 measurable in $L^1(\Omega)$ over $(0,T)$ (see also~\cite[Appendix]{AS}). 
 This is the reason why in the above proof we needed to proceed by approximation for $g \in \calV^*$.
\eddos

\section{Embeddings and compactness results}
\label{subsec:embe}

This section presents some embedding and compactness results useful for the sequel. 
To this end, we start with recalling the so-called {\sl unit ball property}\/ 
(see~\cite[Lemma~2.1.14]{DHHR}), namely for $u\in V$ the following equivalence holds:
\begin{equation}\label{unit:ball}
  \| u \|_V \le 1 \quad \mbox{ if and only if } \quad \modu(u) \le 1. 
\end{equation}
Of course, a similar property holds in $\calV$. Then we have
\begin{lem}[Basic embeddings of Musielak-Orlicz spaces]\label{lem:emb1}
 Let\/ {\rm Assumption~\ref{hp:alpha}} hold. Then, 
 $L^\infty(0,T;V) \hookrightarrow \calV$ and $\calV \hookrightarrow L^1(0,T;V)$ with continuous injections.
\end{lem}
\begin{proof}
We shall prove the first assertion in the case $T\ge 1$, which is slightly more difficult 
than the other case $T\in(0,1)$. Let $u\in L^\infty(0,T;V)$ and choose $\lambda > \| u \|_{L^\infty(0,T;V)}$. 
Note that $u \in L^\infty(0,T;V) \hookrightarrow L^1(0,T;L^1(\Omega))$ can be identified with a 
Lebesgue integrable function (still denoted by $u$) in $Q = \Omega \times (0,T)$ 
(see~\cite[Appendix]{AS}). Moreover, we have
\begin{equation*}
  \modu\left(\frac{u(t)}{\lambda}\right) \le 1 \ \mbox{ for a.e. } t\in (0,T),
\end{equation*}
whence, by Fubini's lemma and convexity,
\begin{equation*}
  \Modu \left(\frac{u}{\lambda T}\right) 
   = \iTT \modu\left(\frac{u(t)}{\lambda T}\right)\,\dit
   \le \frac1T \iTT \modu\left(\frac{u(t)}{\lambda}\right)\,\dit
   \le 1,
\end{equation*}
and consequently, we get $u\in \calV$ and $\| u \|_{\calV} \le \lambda T$. In particular, 
we have $\| u \|_{\calV} \le T\| u \|_{L^\infty(0,T;V)}$, whence follows the first assertion. 
The second assertion follows immediately from \eqref{norm<modu}, i.e., 
integrating both sides of \eqref{norm<modu} with $u=u(t)$ over $(0,T)$, we find that
$$
\int^T_0 \|u(t)\|_V \, \d t \leq \int^T_0 \modu(u(t)) \, \d t + T \quad \mbox{ for } \ u \in \calV.
$$
Here we also used the fact that, for $u \in \calV \subset L^1(Q)$, 
$u(t) := u(\cdot,t)$ is strongly measurable with values in $V$ in $(0,T)$ (see~\cite[Appendix]{AS}).
This completes the proof.
\end{proof}

We can now define the subspace of $\calV$, 
\begin{equation*}
\calV_1:=\left\{ v \in W^{1,1}(0,T;V) \colon v_t \in \calV \right\},
\end{equation*}
which is naturally endowed with the graph norm 
$$
\|v\|_{\calV_1} := \|v\|_{\calV} + \|v_t\|_{\calV}
\quad \mbox{ for } \ v \in \calV_1.
$$
It is immediate to check that $\calV_1$ is a (reflexive) Banach space. Indeed, it is a closed subspace of $\calV\times \calV$. Note that $\calV_1$ is not properly a Sobolev space; nevertheless, based on the above lemma we have
\begin{equation*}
  \calV_1 \hookrightarrow W^{1,1}(0,T;V) \hookrightarrow C^0([0,T];V)
\end{equation*}
and both embeddings are continuous (see Lemma \ref{lem:emb1}).

We conclude this section by presenting a generalized version of the Aubin-Lions compactness lemma, which may have an independent interest as well.
\begin{teor}[Aubin-Lions lemma for Musielak-Orlicz spaces]\label{thm:AL}
 Let\/ {\rm Assumption~\ref{hp:alpha}}\/ hold and let $X$, $Y$ 
 be Banach spaces such that $X \hookrightarrow Y$ compactly and $Y \hookrightarrow V = L^\phj(\Omega)$ continuously. Let $L$ be a bounded subset of $L^\infty(0,T;X)$ and assume that the set $L_t:=\{f_t \colon f\in L\}$ is bounded in $\calV = L^\phj(Q)$. Then $L$ is relatively compact in $C^0([0,T];Y)$.
\end{teor}
\begin{rem}[Comparison with the standard Aubin-Lions-Simon lemma]
According to~\cite[Corollary 4 in \S 8]{Si}, under the same setting of $X$, $Y$ and $V$ as above, the following holds: Let $L$ be a bounded subset of $L^p(0,T;X)$ such that the set $\{u_t \colon u \in L\}$ is bounded in $L^r(0,T;V)$. Then $L$ is precompact in $L^p(0,T;Y)$ if $1 \leq p < \infty$ and $r=1$; in $C^0([0,T];Y)$ if $p=\infty$ and $r>1$. Moreover, this result is optimal; in particular, even if $L$ is bounded in $C^0([0,T];X)$ 
and the set $\{u_t \colon u \in L\}$ is bounded in $L^1(0,T;V)$, then 
$L$ may not be precompact in $L^\infty(0,T;Y)$ (see~\cite[Proposition 4]{Si}).
Recall that $\calV$ is always (continuously) embedded in $L^1(0,T;V)$ (see Lemma \ref{lem:emb1}) 
but it may not be included in $L^r(0,T;V)$ for any $r > 1$. Hence the compact 
embeddings established in~\cite{Si}  do not imply  the assertion of Theorem \ref{thm:AL} directly.
\end{rem}
\begin{proof}
By assumption, there exists $C>0$ such that
\begin{equation*}
  \| u_t \|_{\calV} + \| u \|_{L^\infty(0,T;X)} \le C
   \quad \mbox{ for all } \ u \in L.
\end{equation*}
We can assume that 
$$
\sup_{t \in [0,T]}\|u(t)\|_X \le C \quad \mbox{ for all } \ u \in L
$$ 
without loss of generality.
Let $u$ be any element of $L$ such that $u_t \not\equiv 0$. Let $t\in[0,T)$ and let $h\in(0,1/2)$ be so small that $t+h\in(0,T]$. Choose $n\in\NN$ such that $h\in(2^{-(n+1)},2^{-n}]$.
Then, it follows from \eqref{jlow} that
\begin{equation*}
  \phj\big( x, u(x,t + h) - u(x,t) \big)
    \le K_0^{-n} \phj\left( x,\frac{u(x,t + h) - u(x,t)}{2^{-n}} \right)
\end{equation*}
for a.e.~$x\in \Omega$. Then, recalling that 
$\phj$ is even in the second variable, we deduce that
\begin{align}
\lefteqn{
 \phj\left( x, u(x,t + h) - u(x,t) \right)
    \le K_0^{-n} \phj\left(x, 2^n\textstyle\int_t^{t+h} u_t(x,s) \,\dis \right)
}\no\\
& = K_0^{-n} \phj\left(x,  2^n \left| \textstyle\int_t^{t+h} u_t(x,s) \,\dis\right| \right)\no\\
    & \le K_0^{-n} \phj\left(x, 2^n \textstyle\int_t^{t+h} |u_t(x,s)| \,\dis \right)
   \le K_0^{-n} \phj\left(x, 2^n  \textstyle\int_t^{t+2^{-n}} |u_t(x,s)| \,\dis \right).\no
\end{align}
Integrating both sides over $\Omega$ and subsequently applying Jensen's inequality together with \eqref{K2}, we infer that, 
for  $h$ small enough and correspondingly  $n \in \N$ large enough,
\begin{align}
\lefteqn{
 \modu(u(t + h) - u(t))
     \le K_0^{-n} \io \phj\left(x, 2^n  \textstyle\int_t^{t+2^{-n}} |u_t(x,s)| \,\dis \right) \,\dix 
 \no}\\
   & \le K_0^{-n} \io \left( 2^n \textstyle\int_t^{t+2^{-n}} \phj(x,|u_t(x,s)|) \,\dis \right)\,\dix \no\\
 &     = \left( \frac{K_0}2\right)^{-n} \int_t^{t+2^{-n}} \io \phj(x,|u_t(x,s)|) \,\dix \dis \no\\
 \no
   &
     = \left( \frac{K_0}2\right)^{-n} \int_t^{t+2^{-n}} \modu(u_t(s)) \,\dis
 \no
   \le \left( \frac{K_0}2\right)^{-n} \int_0^T \modu(u_t(s)) \,\dis\\
   & = \left( \frac{K_0}2\right)^{-n} \Modu(u_t)
    \le \left( \frac{K_0}2 \right)^{1+\log_2 h} \Modu\left(\| u_t\|_{\calV}\frac{u_t}{\| u_t\|_{\calV}}\right) \no  \\
 \no
   & \le \left( \frac{K_0}2 \right)^{1+\log_2 h} \left( K^{1 + \log_2\|u_t\|_{\mathcal V}} 
           \vee 1 \right) \Modu\left(\frac{u_t}{\| u_t\|_{\calV}}\right) \no\\
   & \le \left( \frac{K_0}2 \right)^{1+\log_2 h} \left( K^{1 + \log_2 C} \no
\vee 1 \right),
\end{align}
which in particular implies that, for $h > 0$ small enough, the \lhs\ is less than or equal to one, and it follows from \eqref{unit:ball} that $\| u(t+h) - u(t) \|_V \leq 1$. Consequently, repeating a similar argument along with \eqref{K2}, we can deduce that
\begin{align}\no
  \lefteqn{
 1 = \modu\left( \frac{u(t + h) - u(t)}{\| u(t+h) - u(t) \|_V} \right)
 }\\
 & \le K^{1-\log_2 \|u(t+h)-u(t)\|_V} 
\modu(u(t + h) - u(t)) \no\\
 \no 
  & \le  K^{1-\log_2 \|u(t+h)-u(t)\|_V} 
  \left( \frac{K_0}2\right)^{1+\log_2 h} \left[ K^{1+\log_2C} 
\vee 1 \right],
\end{align}
which implies that
\begin{equation}\label{Aubin:15}
  \lim_{h\to 0_+} \| u(t+h) - u(t) \|_V = 0
   \quext{uniformly in }\,L.
\end{equation}
Consequently, we have checked the assumptions of Ascoli's lemma (see, e.g.,~\cite[Lemma 1]{Si}), 
whence $L$ turns out to be precompact in $C^0([0,T];V)$.
The compactness of $L$ in $C^0([0,T];Y)$ is then deduced by noting
that, thanks to Ehrling's lemma, \eqref{Aubin:15} can easily be improved to
\begin{equation*}
  \lim_{h\to 0} \| u(t+h) - u(t) \|_Y = 0
   \quext{uniformly in }\,L.
\end{equation*}
Indeed, for any $\epsi>0$ there exists $C_\epsi>0$ such that
\begin{align*}
  \| u(t+h) - u(t) \|_Y 
    &\le \epsi \| u(t+h) - u(t) \|_X + C_\epsi  \| u(t+h) - u(t) \|_V\no\\
    &\le 2 C \epsi + C_\epsi  \| u(t+h) - u(t) \|_V.
\end{align*}
This completes the proof.
\end{proof}

\section{Chain-rule formula}
\label{subsec:chain}

The chain-rule formula is one of the advantages of formulating equations with subdifferential 
operators and indeed plays a crucial role to establish energy estimates as well as to 
identify weak limits of nonlinear terms (see \S \ref{sec:proof} below). In this section, 
we shall prove a chain-rule formula in the spirit of ~\cite[Lemme~3.3, p.~73]{Br}, which extends,
to the Musielak-Orlicz space setting, our former result~\cite[Prop.~4.1]{AS} dealing
with $L^{p(x)}$-spaces (cf.~Lemma \ref{L:chain-0} below for the standard version of the formula).
\begin{thm}[Chain-rule in Musielak-Orlicz spaces]\label{teo:chain}
 Let\/ {\rm Assumptions~\ref{hp:alpha}} and\/ {\rm \ref{hp:B}} hold. 
 Let $u\in \calV_1$ and $ \eta  \in \calVp$ be such that $u \in D(\calB)$ and  
 $ \eta  \in \calB(u)$. Then the function  $t\mapsto \fhi(u(t))$ is absolutely continuous 
 on $[0,T]$ and the following chain-rule formula holds\/{\rm :}
 \begin{equation}\label{chain}
   \int_s^t \duav{  \eta(\tau) ,  \partial_t u(\tau) }\,\d \tau = \fhi(u(t)) - \fhi(u(s))
    \ \mbox{ for all } \  0\le s \le t \le T.
 \end{equation}
\end{thm}
The proof of the theorem above will be given at the end of this section. 
Indeed, we first need to present some amount of preparatory material under 
Assumption \ref{hp:alpha}. In what follows, we shall treat only the 
functionals $\fhi$ and $\Fhi$, which are supposed to fulfill Assumption \ref{hp:B}. 
However, we shall not use any specific assumptions for such functionals except 
convexity and lower semicontinuity (for instance, (b) of Assumption \ref{hp:B} 
will be never used in this section). Let us begin with introducing a notion of 
Moreau-Yosida regularization of convex functionals. In what follows,
$\lambda\in(0,1)$, intended to go to $0$ in the limit, will denote the 
regularization parameter.
\begin{df}[Moreau-Yosida regularization based on the modular]\label{def:MY}
Let $\lambda\in(0,1)$ Then, the \emph{Moreau-Yosida regularization} of the functional $\fhi : V \to [0,+\infty]$
based on $\modu$ is given by 
 \begin{equation}\label{fhilla}
   \fhi\lla(u):= \min_{v\in V}
    \left[\lambda\modu\left(\frac{v-u}{\lambda}\right) + \fhi(v) \right]
     \quext{for }\, u \in V.
 \end{equation}
Analogously, the \emph{Moreau-Yosida regularization based on} $\Modu$ of 
the functional $\Fhi : \calV \to [0,+\infty]$ is defined as
 \begin{equation*}
   \Fhi\lla(u):= \min_{v\in \mathcal V}
    \left[ \lambda\Modu\left(\frac{v-u}{\lambda}\right) + \Fhi(v) \right]
    \quext{for }\, u \in \mathcal V.
 \end{equation*}
\end{df}
The above definition may remind us of the \emph{Hopf-Lax formula} for Hamilton-Jacobi equations. 
These notions allow us to introduce a generalized notion of {\sl resolvent}\/ of subdifferential 
operators. This will be defined in the following lemma,
where, for simplicity, we just deal with operators defined on~$V$;
however, an analogue in $\calV$ clearly holds as well.
\begin{lem}[Resolvent based on the modular]\label{lemma:res}
For each $u\in V$, $u\lla \in V$ is a minimizer in~\eqref{fhilla}
 if and only if $u\lla \in D(B)$ solves
 \begin{equation}\label{def-J}
  A \left(\frac{u_\lambda - u}{\lambda}\right) 
    + B(u_\lambda) \ni 0 
   \ \mbox{ in }\, \Vp.
 \end{equation}
  Moreover, the minimizer $u\lla \in D(B)$ exists uniquely. In what follows, 
  $u_\lam$ will be denoted by $J_\lam(u)$ and the operator $J_\lam : V \to D(B); u \mapsto u_\lam$ is called a \emph{resolvent based on} $\modu$ of $B$. Furthermore, it holds that
\begin{equation}\label{MY-order}
\fhi(J_\lam(u)) \leq \fhi_\lam(u) \leq \fhi(u) \quad \mbox{ for } \ u \in V.
\end{equation}
If $u \in D(\fhi)$, then we also have
\begin{equation}\label{J-conv}
J_\lam(u) \to u \quad \mbox{ strongly in } V \mbox{ as } \lam \to 0_+,
\end{equation}
and hence,  by the lower semicontinuity of $\fhi$ in $V$, 
\begin{equation}\label{MO-conv}
\fhi(J_\lam(u)) \to \fhi(u) \ \mbox{ as } \lam \to 0_+.
\end{equation}
Moreover, $\fhi_\lambda$ is G\^ateaux differentiable on $V$ and the G\^ateaux 
derivative $\partial_\Omega (\fhi_\lambda)$ of the Moreau-Yosida regularization $\fhi_\lambda$ 
coincides with the Yosida approximation $(\partial_\Omega \fhi)_\lambda$ 
of $\partial_\Omega \fhi$ {\rm (}hence we shall simply denote by $\partial_\Omega \fhi_\lambda$ 
both of them{\rm )}.
\end{lem}
\begin{proof}
Let $u_\lam$ satisfy \eqref{def-J}. We shall prove $u_\lam$ being a minimizer. Since $\alpha(x,r)$ is odd in $r$, relation \eqref{def-J} can be equivalently rewritten as 
\begin{equation*}
  \alpha\left(\cdot\,,\frac{ u - u_\lambda }{\lambda}\right) = b_\lambda, \quad b_\lambda \in B(u_\lambda) \ \mbox{ a.e.~in } \Omega
\end{equation*}
(see also Lemma \ref{lem:sing}). 
Then,  for any $v \in D(\fhi)$,  testing the above by $v - u\lla$ and recalling the definition of subdifferential, we infer that
\begin{align*}\no
  \fhi(v) - \fhi(u\lla) 
   & \ge \duav{ b\lla, v - u\lla }
    = \io \alpha\left(\cdot\,,\frac{ u - u_\lambda}{\lambda}\right) ( v - u\lla ) \, \d x\\
   & = \lambda \io \alpha\left(\cdot\,,\frac{ u- u_\lambda}{\lambda}\right) 
      \left( \frac{u - u\lla}\lambda - \frac{ u - v }\lambda \right) \, \d x \no\\
&     \ge \lambda \modu\left( \frac{u - u\lla}\lambda \right) 
  - \lambda \modu\left( \frac{u - v}\lambda \right),
\end{align*}   
whence $u_\lam$ achieves the minimum of \eqref{fhilla}. 

Conversely, let $u_\lam$ be a minimizer of \eqref{fhilla}. Then, for any $t\in(0,1)$ and $v\in V$, one has
\begin{align*}
\lefteqn{
  \lam \modu([u_\lam-u]/\lam) + \fhi(u_\lam)
}\no\\
 & \le \lam \modu([t u_\lam + (1-t) v-u]/\lam) + \fhi(t u_\lam + (1-t) v) \no\\
 & \le \lam \modu([t u_\lam + (1-t)v - u]/\lam) + t \fhi(u_\lam) + (1-t) \fhi(v).
\end{align*}
Hence, 
\begin{equation*}
  \fhi(u_\lam) -  \fhi(v) 
    \le \lam \frac{ \modu( [u_\lam-u]/\lam + (1-t)[v - u_\lam]/\lam ) - \modu([u_\lam-u]/\lam) }{1-t}.
\end{equation*}
Applying the definition of $A$ again, we deduce that
\begin{equation*}
  \fhi(u_\lam) - \fhi(v)
    \le \duav{ A( [u_\lam-u]/\lam + (1-t)[v - u_\lam]/\lam ), v - u_\lam }.     
\end{equation*}
Now, by virtue of the demicontinuity of $A$ (see Lemma \ref{L:A-conti}),
$$
\left\langle A( [u_\lam-u]/\lam + (1-t)[v - u_\lam]/\lam ), v-u_\lam \right\rangle
\to \left\langle A([u_\lam-u]/\lam), v-u_\lam \right\rangle \ \mbox{ as } t \nearrow 1,
$$
which implies
\begin{equation*}
  \fhi(u_\lam) - \fhi(v)
     \le \duav{ -A([u_\lam-u]/\lam), u_\lam-v }
     \ \mbox{ for }  v \in V. 
\end{equation*}
Recalling that $B=\deo \fhi$, the above relation can be equivalently rewritten as 
$u_\lam \in D(B)$ and $-A([u_\lam-u]/\lam) \in B(u_\lam)$, which is the assertion.

The uniqueness of minimizers follows from the strict convexity of the functional $v \mapsto \lambda\modu([v-u]/\lambda)+\fhi(v)$ for each $u \in V$ fixed. 
Moreover, \eqref{MY-order} follows immediately from the definition, and moreover, \eqref{J-conv} 
and \eqref{MO-conv} can be proved as in the standard setting (see, e.g.,~\cite{BCP},~\cite{barbu}). 
Finally, as in~\cite{barbu}, one can verify the G\^ateaux differentiability of $\fhi_\lam$ in $V$ 
as well as the coincidence $\partial_\Omega (\fhi_\lam) = (\partial_\Omega \fhi)_\lambda$. 
\end{proof}

Now, the \emph{Yosida approximation} $B\lla : V \to \Vp$ \emph{based on} $\Modu$ of the operator $B$ is defined by
\begin{equation}\label{def:Alt}
  B\lla (u) := A \left(\frac{u - J\lla(u)}{\lambda}\right)
    \quext{for }\, u \in V
\end{equation}
with $A=\partial_\Omega\modu$.
 Since $A$ is bounded from $V$ to $\Vp$ (see Lemma \ref{L:A-conti}), so is $B\lla$ for $\lam>0$ fixed. Indeed, $J_\lam : V \to V$ is also bounded for $\lam > 0$. 
Hence \eqref{def-J} yields
\begin{equation}\label{incl:yos}
  B\lla (u) \in B( J\lla(u) )
    \quext{for all }\, u \in V.
\end{equation}
The following lemma is also useful:
\begin{lem}[Convergence of Yosida approximations]\label{L:YA-conv}
For each $u \in D(B)$, the family $(B_\lam(u))$ is bounded in $V^*$ as $\lam \to 0_+$, 
and moreover, there exist a sequence $\lam_n \to 0_+$ and $b_0 \in V^*$ such that 
$B_{\lam_n}(u) \to b_0$ weakly in $V^*$ as $\lam_n \to 0_+$ and $b_0 \in B(u)$. 
\end{lem}

\begin{proof}
Recalling $B_\lam(u) \in B(J_\lam(u))$, we have
$$
\langle b - B_\lam(u), u - J_\lam(u) \rangle \geq 0 \quad \mbox{ for any }\ b \in B(u),
$$
which implies
\begin{align*}
 \langle b , [u-J_\lam (u)]/\lam \rangle &\geq \langle B_\lam(u), [u-J_\lam(u)]/\lam \rangle\\
&= \modu([u-J_\lam(u)]/\lam) + \modud(B_\lam(u)).
\end{align*}
Here we used the relation $B_\lam(u) = A([u-J_\lam(u)]/\lam)$ and the Fenchel-Moreau identity. 
Thus, thanks to Lemma \ref{lemma:1b}, we find that $([u-J_\lam(u)]/\lam)$ is bounded in $V$. 
Furthermore, by virtue of Lemma \ref{L:bdd-equiv}, it also follows from the inequality above that
$$
\|B_\lam(u)\|_{V^*} \leq C.
$$
Hence one can take a sequence $\lam_n \to 0_+$ and $b_0 \in V^*$ such that $B_{\lam_n}(u) \to b_0$ weakly in $V^*$ as $\lam \to 0_+$. Recalling that $J_{\lam_n}(u) \to u$ strongly in $V$ and employing the demiclosedness of maximal monotone operators, we conclude that $b_0 \in B(u)$.
\end{proof}

The resolvent $\calJ_\lambda:\calV\to D(\calB)$ of $\deq\Fhi$ based on $\Modu$ 
is defined in a similar way. The \emph{Yosida approximation} $\calB\lla : \calV \to \calVp$ 
of $\calB$ based on $\Modu$ is defined analogously with \eqref{def:Alt}, is bounded from $\calV$ to $\calVp$,
and satisfies a relation similar to \eqref{incl:yos}. Furthermore, an analogue of Lemma \ref{L:YA-conv} for $\calA$ also holds true. We next have

\begin{lem}[Resolvent and Yosida approximation of $\calB$]\label{lemma:comm}
For every  $u\in \calV$, it holds that $[\calJ\lla(u)](t) = J\lla(u(t))$
 and $[\calB\lla(u)](t) = B\lla(u(t))$ for almost every
 $t\in(0,T)$. 
\end{lem}

\begin{proof}
Let $u \in \mathcal V$ be fixed. Then there uniquely exists $u_\lambda \, (= \calJ_\lambda(u)) \in D(\calB)$ such that
 $$
 \calA\left(\dfrac{u_\lambda-u}\lambda\right)+\calB(u_\lambda)\ni 0 \ \mbox{ in } \calV^*,
 $$
 which together with Lemmas \ref{lem:sing} and \ref{lemma:3} implies
 $$
 A\left(\dfrac{u_\lambda(t)-u(t)}\lambda\right)+B(u_\lambda(t))\ni 0 \ \mbox{ for a.e. } t \in (0,T).
 $$
 Since $u(t)$ and $u_\lambda(t)$ belong to $V$ for a.e.~$t \in (0,T)$, we deduce from the uniqueness of solution to \eqref{def-J} that
 $$
 [\calJ_\lambda(u)](t) = u_\lambda(t) = J_\lambda(u(t)) \ \mbox{ for a.e. } t \in (0,T).
 $$
 Hence in particular, one has $J_\lambda(u(\cdot)) \in \calV$. Moreover, we obtain
 $$
 [\calB_\lambda(u)](t) = B_\lambda(u(t)) \ \mbox{ for a.e. } t \in (0,T).
 $$
 This completes the proof.
\end{proof}

In order to prove Theorem \ref{teo:chain}, we need to use a
``standard'' version of the chain-rule formula, which is reported, together
with its proof, for the convenience of the reader. 
\begin{lem}[Standard chain-rule formula for subdifferentials]\label{L:chain-0}
Let $B$ a Banach space and denote by $B^*$ its dual space with the duality pairing $\langle \cdot, \cdot \rangle_B$. Let $\phi : B \to (-\infty,+\infty]$ be a lower semicontinuous convex functional such that $\phi \not\equiv +\infty$ and let $\partial \phi : B \to 2^{B^*}$ be the subdifferential operator of $\phi$. Let $u \in W^{1,1}(0,T;B)$ be such that $u(t) \in D(\partial \phi)$ for a.e.~$t \in (0,T)$ and assume that there exists $g \in L^\infty(0,T;B^*)$ such that $g(t) \in \partial \phi(u(t))$ for a.e.~$t \in (0,T)$. Then the function $t \mapsto \phi(u(t))$ is absolutely continuous on $[0,T]$, and moreover,
 $$
 \dfrac{\d}{\d t} \phi(u(t)) = \langle h, \partial_t u(t) \rangle_B \ \mbox{ for any } h \in \partial \phi(u(t)) \mbox{ and a.e. } t \in (0,T).
 $$
\end{lem}

\begin{proof}
We shall first prove that $t \mapsto \phi(u(t))$ is absolutely continuous on $[0,T]$. By the definition of subdifferential, we see that
 $$
 \phi(u(t))-\phi(u(s)) \leq \langle g(t),u(t)-u(s) \rangle_B
 \leq \|g\|_{L^\infty(0,T;B^*)} \int^t_s \|\partial_t u(\tau)\|_B \, \d \tau 
 $$
 for any $0 \leq s < t \leq T$. Repeating the same argument above, we can also derive that 
$$
\phi(u(s)) - \phi(u(t)) \leq \|g\|_{L^\infty(0,T;B^*)} \int^t_s \|\partial_t u(\tau)\|_B \, \d \tau. 
$$
Thus we obtain the absolute continuity of the function $t \mapsto \phi(u(t))$ on $[0,T]$, 
since $\|\partial_t u(\cdot)\|_B$ belongs to $L^1(0,T)$. Moreover, let $h \in \partial \phi(u(t))$. 
Then one can derive that, for $\delta > 0$ small enough, 
 $$
 \phi(u(t+\delta))-\phi(u(t)) \geq \langle h, u(t+\delta)-u(t) \rangle_B \ \mbox{ for a.e. } t \in (0,T).
 $$
 Dividing both sides by $\delta>0$ and taking a limit as $\delta \to 0_+$, one obtains
 $$
 \dfrac{\d}{\d t} \phi(u(t)) \geq \langle h, \partial_t u(t) \rangle_B \ \mbox{ for a.e. } t \in (0,T).
 $$
 The inverse inequality can also be checked by repeating the same argument with $\delta<0$ 
 and by passing to the limit as $\delta \to 0_-$.
\end{proof}

Thanks to the preparatory material established so far, we are now able to prove Theorem \ref{teo:chain} by 
 basically following the lines of the proof of the standard chain-rule formula 
(see, e.g.,~\cite{Br}). However, we stress that,  in our setting, the argument strongly
relies on the modular-based versions of resolvent, Yosida approximation and Moreau-Yosida regularization 
introduced before. Conversely, our methods would not work, at least in a straightforward way,
if the standard notions of these objects (i.e.~those defined in~\cite{barbu,BCP}) would be 
considered instead. 
\begin{proof}[Proof of Theorem \ref{teo:chain}]
Let $u \in \calV_1 \cap D(\mathcal B)$. Then we observe from Lemma \ref{lem:emb1} that $u \in W^{1,1}(0,T;V)$. Moreover, it follows that $\partial_\Omega \fhi_\lambda(u(\cdot)) \in L^\infty(0,T;\Vp)$ for $\lam > 0$ due to the boundedness of $\partial_\Omega \fhi_\lambda : V \to V^*$ and $u : [0,T] \to V$. Hence, thanks to Lemma \ref{L:chain-0}, we observe that $t \mapsto \fhi_\lam(u(t))$ is absolutely continuous on $[0,T]$, and moreover,
$$
\dfrac{\d}{\d t} \fhi_\lambda(u(t)) = \langle \partial_\Omega \fhi_\lambda(u(t)), \partial_t u(t) \rangle_V 
\ \mbox{ for a.e. } t \in (0,T).
$$
Integrate both sides over $(s,t)$ with $0<s<t<T$. Then we have
\begin{align*}
\fhi_\lambda(u(t)) - \fhi_\lambda(u(s)) 
&= \int^t_s \langle \partial_\Omega \fhi_\lambda(u(\tau)), \partial_t u(\tau) \rangle_V \, \d \tau \\
&= \duaV{\partial_Q \Fhi_\lam(u), \partial_t u}.
\end{align*}
Here we recall that $\duaV{\cdot,\cdot}$ stands for the duality pairing between $\mathcal V$ and $\mathcal V^*$. Here we also used the analogue of Lemma \ref{lemma:comm} on subintervals and Fubini's lemma to verify the last equality. 
Take a limit as $\lambda \to 0_+$ and note by Lemma \ref{lemma:res} that
$$
  \fhi_\lambda(w) \to \fhi(w) \ \mbox{ for } \ w \in D(\fhi).
$$
Moreover, for $w \in D(\calB)$, one can take $b \in \calB(w)$ such that, 
 at least for a non-relabeled subsequence of $\lambda \to 0_+$, 
$$
  \partial_Q \Fhi_\lambda(w) \to b \ \mbox{ weakly in } \calVp
$$
(see Lemma \ref{L:YA-conv}). Hence we obtain
  $$
   \fhi(u(t)) - \fhi(u(s)) = \int^t_s \langle b(\tau), \partial_t u(\tau) \rangle_V \, \d \tau.
  $$
  Since $\partial_t u \in \calV$ and $b \in \calV^*$ (i.e., $\langle b(\cdot), \partial_t u(\cdot) \rangle \in L^1(0,T)$), the function $t \mapsto \fhi(u(t))$ turns out to be absolutely continuous on $[0,T]$. 
  Then, repeating the same argument as in the proof of Lemma \ref{L:chain-0}, we can also verify \eqref{chain} for a general $\eta \in \calB(u)$.
 \end{proof}

\section{Proof of Theorem \ref{teo:main}}
\label{sec:proof}

We are now in a position to prove Theorem \ref{teo:main}. We shall follow a more or less standard strategy for proving existence of strong solutions (see~\cite{AS}) but based on the devices developed so far specifically for the present Musielak-Orlicz setting.

\subsection{Approximation}

In this subsection, we shall construct approximate solutions in terms of time-discretization.
To this end, we fix $K \in \N$ and set $\tau := T/K > 0$, and then, we consider the following minimizing problem,
$$
\mbox{Minimize } \ J_k(w) := \tau \modu \left( \dfrac{w-u_{k-1}}\tau \right) + 
\fhi(w) - \langle f_k, w \rangle_V \ \mbox{ over } \ w \in V,
$$
where $f_k \in \Vp$ is given by
$$
f_k := \dfrac 1 \tau \int^{k\tau}_{(k-1)\tau} f(t) \, \dit,
$$
for $k = 1,\ldots,K$. Then, for each $k$, there exists a unique minimizer $u_k \in D(B)$ of $J_k$, and moreover, $u_k$ solves
\begin{equation}\label{disc-k}
A\left( \dfrac{u_k-u_{k-1}}\tau \right) + B(u_k) \ni f_k \ \mbox{ in } \Vp.
\end{equation}
Indeed, existence and uniqueness of the minimizer can be proved by the use of Direct Method of Calculus of Variation and the strict convexity of $J_k$, respectively. Moreover, the Euler-Lagrange equation \eqref{disc-k} also follows from the sum-rule of subdifferentials, i.e.,
$$
\partial_\Omega \left[ \tau \modu \left(\frac{\cdot-u_{k-1}}\tau \right) + \fhi(\cdot)\right]
= \partial_\Omega \modu \left(\frac{\cdot-u_{k-1}}\tau\right) + \partial_\Omega \fhi(\cdot)
\ \mbox{ for each } \ k.
$$
Indeed,  the above holds whenever  
the domain of $A$ coincides with the whole of $V$ (see, e.g.,~\cite{barbu}).

We next define the \emph{piecewise linear interpolant} $u_\tau : [0,T] \to V$ and \emph{piecewise constant interpolant} $\bar u_\tau : [0,T] \to V$ by
\begin{alignat*}{4}
 u_\tau(t) &= \dfrac{t-(k-1)\tau}\tau u_k + \dfrac{k\tau - t}\tau u_{k-1}, \quad \bar u_\tau(t) = u_k \\
   &\qquad \quad \mbox{ for } \ t \in ((k-1)\tau, k\tau] \ \mbox{ and } \ k = 1,\ldots,K,\\
 u_\tau(0) &=\bar u_\tau(0)=u_0.
\end{alignat*}
Then $u_\tau \in W^{1,\infty}(0,T;V) \subset \calV_1$ and $\bar u_\tau \in L^\infty(0,T;V) \subset \calV$ (see Lemma \ref{lem:emb1}). Furthermore, it follows that
\begin{equation*}
A(\partial_t u_\tau(t))+ \eta_\tau(t) = \bar f_\tau(t), \quad \eta_\tau(t) \in B(\bar u_\tau(t)) \ \mbox{ in } V^* \ \mbox{ for a.e. } t \in (0,T)
\end{equation*}
and also
\begin{equation}\label{aprx-EQ}
\calA(\partial_t u_\tau)+ \eta_\tau = \bar f_\tau, \quad \eta_\tau \in \calB(\bar u_\tau) \ \mbox{ in } \calVp,
\end{equation}
where $\eta_\tau := \bar f_\tau - \calA(\partial_t u_\tau)$ and $\bar f_\tau$ stands for the piecewise constant interpolant of $\{f_k\}_{k=1,2,\ldots,K}$ defined as above. Here we remark that
\begin{lem}
It holds that
\begin{equation}\label{conv-f}
\bar f_\tau \to f \quad \mbox{ strongly in } \calV^*.
\end{equation}
\end{lem}

\begin{proof}
By Jensen's inequality, we observe that
$$
\phj^*(x,f_k(x)) \leq \frac 1 \tau \int^{k\tau}_{(k-1)\tau} \phj^*(x, f(x,t)) \, \d t \quad \mbox{ for } \ k=1,2,\ldots,K,
$$
which implies
\begin{equation}\label{intpl-est}
\Modud(\bar f_\tau) \leq \Modud(f) \quad \mbox{ for } \ f \in \calV^*.
\end{equation}
Due to the local integrability of $\phj^*$, the set $S(Q)$ of simple functions is dense in $E^{\phj^*}(Q)$, which coincides with $L^{\phj^*}(Q)$ by virtue of the $\Delta_2$-condition for $\phj^*$ (see~\cite[Lemma 2.5.9]{DHHR} and \S \ref{Ss:MO}). Hence $f \in \calV^* \simeq L^{\phj^*}(Q)$ can be approximated by a sequence $(f_n)$ of simple functions such that
$$
\|f - f_n\|_{L^{\phj^*}(Q)} < \frac 1n.
$$
Furthermore, note as in \eqref{intpl-est} that
\begin{align*}
\Modud \left(\bar f_\tau - \overline{(f_n)}_\tau \right)
&= \Modud \left( \overline{(f - f_n)}_\tau\right)
\leq \Modud(f-f_n),
\end{align*}
which along with Lemma \ref{L:equiv} implies
$$
\bar f_\tau - \overline{(f_n)}_\tau \to 0 \ \mbox{ in } L^{\phj^*}(Q) \quad \mbox{ as } \ n \to +\infty
$$
uniformly for $\tau > 0$. Furthermore, noting that $f_n \in S(Q)$, with the aid of Lebesgue's dominated convergence theorem, one can verify that
$$
f_n - \overline{(f_n)}_\tau \to 0 \ \mbox{ in } L^{\phj^*}(Q) \quad \mbox{ as }  \ \tau \to 0_+.
$$
Indeed, it holds that $f_n - \overline{(f_n)}_\tau \to 0$ strongly in $L^r(Q)$ as 
$\tau \to  0_+ $ for any $r \in [1,+\infty)$ (see~\cite[Appendix]{AS}), and hence $f_n - \overline{(f_n)}_\tau \to 0$ 
a.e.~in $Q$, which implies that 
$\phj^*(x, \overline{(f_n)}_\tau(x,t)-f_n(x,t))  \to 0 $ for a.e.~$(x,t) \in Q$ as $\tau \to  0_+ $.
Moreover, note that $\phj^*(x,\overline{(f_n)}_\tau(x,t)-f_n(x,t)) \leq \phj^*(x, \|\overline{(f_n)}_\tau-f_n\|_{L^\infty(Q)}) \leq \phj^*(x,2\|f_n\|_{L^\infty(Q)}) \in L^1(Q)$ by (ii) of (b) in Assumption \ref{hp:alpha}. 
Combining all these facts, we finally obtain \eqref{conv-f}.
\end{proof}

\subsection{A priori estimates}

Test \eqref{disc-k} by $(u_k-u_{k-1})/\tau$ to see that
$$
\left\langle A\left( \dfrac{u_k-u_{k-1}}\tau \right), \dfrac{u_k-u_{k-1}}\tau \right\rangle + \left\langle \eta_k, \dfrac{u_k-u_{k-1}}\tau \right\rangle = \left\langle f_k, \dfrac{u_k-u_{k-1}}\tau \right\rangle.
$$
By the use of the Fenchel-Moreau identity as well as the definition of $\partial_\Omega \fhi$, we have
\begin{align}\label{disc-en}
\modu \left( \dfrac{u_k-u_{k-1}}\tau \right) + \frac{\fhi(u_k)-\fhi(u_{k-1})}\tau
\leq \left\langle f_k , \dfrac{u_k-u_{k-1}}\tau \right\rangle.
\end{align}
Multiplying both sides by $\tau$ and summing them up for $k=1,2,\ldots,K$, one obtains
\begin{align*}
 \int^T_0 \modu \left( \partial_t u_\tau(t) \right) \, \d t + \fhi(u_K)
 \leq \fhi(u_0) + \|\bar f_\tau\|_{\calV^*} \|\partial_t u_\tau\|_{\calV}.
\end{align*}
Hence, thanks to Lemma \ref{lemma:1b}, it follows that
\begin{equation}\label{e:u_t}
\|\partial_t u_\tau\|_{\calV} \leq C.
\end{equation}
The boundedness of $\calA:\calV \to \calV^*$ (see Lemma \ref{L:A-conti}) yields
\begin{equation*}
\|\calA(\partial_t u_\tau)\|_{\calV^*} \leq C.
\end{equation*}
Hence recalling \eqref{aprx-EQ}, we observe that
\begin{equation}\label{e:eta}
\|\eta_\tau\|_{\calV^*} \leq C. 
\end{equation}
Furthermore, since $\calV \hookrightarrow L^1(0,T;V)$ continuously, we derive from \eqref{e:u_t} that
$$
\sup_{t \in [0,T]} \|u_\tau(t)\|_V \leq C,
$$
which together with the definitions of interpolants implies
$$
\sup_{t \in [0,T]} \|\bar u_\tau(t)\|_V \leq C.
$$
Summing up \eqref{disc-en} again 
(now for $k \in \{1,2,\ldots,m\}$, where $1\le m\le K$), we can deduce that
$$
\max_{k \in \{1,2,\ldots,K\}} \fhi(u_k) \leq C,
$$
that is,
$$
\sup_{t \in [0,T]} \fhi(\bar u_\tau(t)) \leq C.
$$
Moreover, by convexity, it follows that
$$
\sup_{t \in [0,T]} \fhi(u_\tau(t)) \leq C.
$$
Therefore using (b) of Assumption \ref{hp:B} and Lemma \ref{L:bdd-equiv}, one finds that
\begin{equation}
\sup_{t \in [0,T]} \left( \|\bar u_\tau(t)\|_X + \| u_\tau(t)\|_X \right) \leq C.
\label{e:uX}
\end{equation}

\subsection{Convergence}
The a priori estimates established so far enable us to pass to the limit as  $K\to+\infty$
(or, equivalently, $\tau\to 0_+$: both notations will be used on occurrence). 
Indeed, due to the reflexivity of $V$ and of $\calV$, the estimates above imply that, for a (not relabeled)
subsequence of $(n)$, the following convergence relations hold true:
\begin{alignat}{4}
  \bar u_\tau  &\to \bar u \quad &&  \mbox{weakly in }\,L^\infty(0,T;V),\nonumber\\
\label{cof:12}
\partial_t u_\tau &\to \partial_t u \quad &&\mbox{ weakly in }\,\calV,\\
\no 
\calA(\partial_t u_\tau) &\to \xi \quad &&\mbox{ weakly in }\,\calVp
\end{alignat}
for some limits $u, \bar u \in \calV$ and $\xi \in \calV^*$. One can prove in a standard manner that 
$u = \bar u$ by using the a priori estimate for $(\partial_t u_\tau)$ (see, e.g.,~\cite{AS}
 for details). Comparing terms of \eqref{aprx-EQ}, we can then deduce that
\begin{equation}\label{cof:14}
   \eta_\tau \to f - \xi =: \eta \quad \quext{weakly in }\,\calVp,
\end{equation}
whence the following relation holds: 
\begin{equation}\label{eq:limf}
  \xi + \eta = f \quext{in }\,\calVp.
\end{equation}
Hence, to complete the proof, we need to identify the limits
$\xi$ and $\eta$. Thanks to \eqref{e:u_t} and \eqref{e:uX}, applying Lemma~\ref{thm:AL} 
to $L=(u_\tau)$, we can deduce that
\begin{equation}\label{strong:un}
  u_\tau \to u \quad \mbox{ strongly in }  C^0([0,T];V),
\end{equation}
and in particular,
\begin{equation*}
  u_\tau \to u \quad \mbox{ strongly in } \calV,
\end{equation*}
which along with the a priori estimate for $(\partial_t u_\tau)$ enables us to prove that
$$
\bar u_\tau \to u \quad \mbox{ strongly in } \calV.
$$
Combining this fact with \eqref{cof:14} and applying the 
maximal monotonicity (more precisely, demiclosedness) of $\calB:\calV \to 2^{\calVp}$ we can conclude that
\begin{equation*}
u \in D(\calB), \quad  \eta \in \calB(u) \ \mbox{ in }\,\calVp.
\end{equation*}
Thus we have identified the limit $\eta$.

Finally, in order to identify $\xi$, we test \eqref{disc-k}
by $(u_k-u_{k-1})/\tau$ and integrate it in time to obtain
\begin{equation*}
  \duaV{\calA(\partial_t u_\tau),\partial_t u_\tau}
   \leq \duaV{\bar f_\tau,\partial_t u_\tau} - \fhi(\bar u_\tau(T)) + \fhi(u_0).
\end{equation*}
Then, taking the  supremum limit  as $\tau \to 0_+$ and using \eqref{conv-f} and \eqref{cof:12},
we infer that
\begin{equation*}
  \limsup_{\tau \to 0_+} \duaV{\calA(\partial_t u_\tau),u_\tau} 
   \leq \duaV{f,\partial_t u} - \liminf_{n\to+\infty} \fhi(\bar u_\tau(T)) + \fhi(u_0).
\end{equation*}
Then, using \eqref{strong:un} with the lower semicontinuity of $\fhi$ in $V$, 
 noting also that $\bar u_\tau(T) = u_K = u_\tau(T) \to u(T)$ strongly in $V$, 
we get 
\begin{equation}\label{form:16}
  \limsup_{\tau \to 0_+} \duaV{\calA(\partial_t u_\tau),\partial_t u_\tau} 
   \le \duaV{ f, \partial_t u } - \fhi(u(T)) + \fhi(u_0).
\end{equation}
On the other hand, since $\eta(t) \in B(u(t))$ for a.e.~$t \in (0,T)$,
with the aid of the chain-rule formula developed in Theorem \ref{teo:chain}, we can deduce that
\begin{equation}\label{form:17}
  - \fhi(u(T)) + \fhi(u_0) 
   = - \int_0^T \duav{\eta(\tau),\partial_t u(\tau)}\,\d \tau
   = - \duaV{\eta,\partial_t u}.
\end{equation}
Hence, by virtue of \eqref{eq:limf}, it follows from \eqref{form:16} that
\begin{equation*}
  \limsup_{\tau \to 0_+} \duaV{\calA(\partial_t u_\tau),\partial_t u_\tau}
   \le \duaV{f,\partial_t u} - \duaV{\eta,\partial_t u}
   = \duaV{\xi,\partial_t u},
\end{equation*}
whence, by the maximal monotonicity of $\calA:\calV\to\calVp$, we finally obtain 
$$
\partial_t u \in D(\calA), \quad \xi=\calA(\partial_t u) \ \mbox{ in } \calV^*.
$$
Finally, we remark that the weak continuity property $u \in C_w([0,T];X)$ in \eqref{rego:u} follows in a standard way by combining the information
contained in \eqref{e:uX} and in \eqref{strong:un} (see~\cite[Lemma 8.1]{LM1}). This concludes the proof. \qed 

\section{Generalization}\label{sec:gen}

This short section is devoted to a generalization of Theorem \ref{teo:main}. 
Namely, we are concerned with the doubly-nonlinear inclusion,
\begin{equation}\label{gen-eq}
 M(u_t(\cdot,t)) + B(u(\cdot,t)) \ni f(\cdot,t) \ \mbox{ in } V^*, \quad 0 < t < T,
\end{equation}
where $B : V \to 2^{V^*}$ fulfills Assumption \ref{hp:B} and $M : V \to 2^{V^*}$ satisfies the following
\begin{assu}[Hypotheses on $M$]\label{hp:M}{\rm
Let $M : V \to 2^{V^*}$ be maximal monotone with domain $D(M) = V$. In addition, the following (i) and (ii) hold:
\begin{enumerate}
 \item[(i)] There exist constants $\alpha_0 > 0$ and $C_1 \geq 0$ such that
$$
\alpha_0 \modu(u) \leq \langle m, u \rangle + C_1 \quad \mbox{ for } \ u \in V, \ m \in M(u).
$$ 
 \item[(ii)] There exists a constant $C_2 \geq 0$ such that
$$
\modud(m) \leq C_2 \left( \modu(u) + 1 \right)  \quad \mbox{ for } \ u \in V, \ m \in M(u).
$$
\end{enumerate}
}\end{assu}

Now, our result reads,
\begin{thm}[Generalization]\label{T:gen}
Let {\rm Assumptions \ref{hp:alpha}, \ref{hp:B}} and {\rm \ref{hp:M}} hold. Then for any $f$ 
and $u_0$ satisfying \eqref{hp:u0} and \eqref{hp:f}, the Cauchy problem for \eqref{gen-eq} 
along with the initial condition $u|_{t=0}=u_0$ admits at least one solution $u$ enjoying 
the regularity \eqref{rego:u} and satisfying 
the energy identity \eqref{energy} with $\alpha(\cdot,u_t)$ replaced by a section $m \in M(u_t)$.
\end{thm}

The theorem mentioned above can be regarded as a generalization of the results in~\cite{Colli,CV} 
to the Musielak-Orlicz setting.

\begin{proof}
We shall modify the proof of Theorem \ref{teo:main} given in Section \ref{sec:proof}. 
A first modification is made for approximation: we also introduce a similar discretization to \eqref{disc-k}, that is,
\begin{equation}\label{disc-gen}
M \left( \frac{u_k-u_{k-1}}\tau \right) + B(u_k) \ni f_k \ \mbox{ in } V^*
\end{equation}
for $k = 1,\ldots,K$. It has no longer variational structure; however, 
it follows from (i) of Assumption \ref{hp:M} along with \eqref{bounded2} that
$M$ is coercive, i.e.,
$$
\lim_{\|u\|_V \to +\infty} \frac{\langle m, u \rangle}{\|u\|_V} = +\infty
\quad \mbox{ for } \ m \in M(u).
$$
Moreover, the sum $M + B$ turns out to be maximal monotone in $V \times V^*$ due to the fact that $D(M)=V$ (see~\cite{BCP},~\cite[Chap.\,II, Theorem 1.7]{barbu}). Combining all these facts, one can verify 
that $M+B$ is surjective from $V$ to $V^*$. Hence one can assure the existence of a solution 
$u_k \in D(B)$ to \eqref{disc-gen} for $k=1,\ldots,K$.

Thanks to Assumption \ref{hp:M}, a priori estimates can be established similarly as before. Indeed, \eqref{disc-en} will be modified as
\begin{align*}
\alpha_0 \modu \left( \dfrac{u_k-u_{k-1}}\tau \right) + \frac{\fhi(u_k)-\fhi(u_{k-1})}\tau
\leq \left\langle f_k, \dfrac{u_k-u_{k-1}}\tau \right\rangle + C_1.
\end{align*}
Hence \eqref{e:u_t} follows by the use of Lemma \ref{lemma:1b}. Moreover, (ii) of Assumption \ref{hp:M}
together with Lemma \ref{L:bdd-equiv} implies that, for any strongly measurable $m : (0,T) \to V$ such that $m(\cdot) \in M(\partial_t u_\tau(\cdot))$ a.e.~in $(0,T)$, there holds
$$
 m\in \calV^*, \qquad \|m\|_{\calV^*} \leq C, 
$$
with $C>0$ depending only on the parameters of the system. In particular, $m_\tau(\cdot) := \bar f_\tau(\cdot) - \eta_\tau(\cdot) \in M(\partial_t u_\tau(\cdot))$
can be estimated in $\calV^*$ uniformly with respect to~$\tau$. 
Hence recalling \eqref{disc-gen}, we get \eqref{e:eta}. Furthermore, all the other uniform 
estimates follow similarly. As for the final step for convergence, we introduce a monotone 
operator $\calM : \calV \to 2^{\calV^*}$ defined as follows: for $u \in \calV$ and $f \in \calV^*$,
$$
[u,f] \in \calM \quad \stackrel{\text{define}}\Leftrightarrow \quad
[u(t), f(t)] \in M \ \mbox{ for a.e. } t \in (0,T).
$$
To complete the proof, we need the maximality of $\calM$ in $\calV \times \calV^*$, which 
enables us to identify the weak limit of $m_\tau \in \calM(\partial_t u_\tau)$ 
as $\tau \to 0_+$ as in the proof of Theorem \ref{teo:main}. Indeed, the maximality can be 
proved as in the proof of Lemma \ref{lemma:3}, where the demiclosedness of the operator 
$T := (A + M)^{-1}$ is needed (see Lemma \ref{lemma:3a}).
\end{proof}

\section{Applications}
\label{sec:ex}

In this section, we shall present a number of concrete doubly nonlinear PDEs to which 
the abstract results developed so far can be applied. To do so, the following lemma is needed 
(see~\cite[Theorem 2.8.1]{DHHR} and~\cite{Mu83}):
\begin{lem}[Embeddings among Musielak-Orlicz spaces]\label{L:MO-embd}
Let $\Omega \subset \R^N$ be a bounded open set and let $\theta$ and $\phj$ be generalized $\Phi$-functions in $\Omega$. Then the embedding
$$
L^\theta(\Omega) \hookrightarrow L^\phj(\Omega)
$$
is continuous if and only if there exist a constant $\lam_0>0$ and a function $h \in L^1(\Omega)$ satisfying $\|h\|_{L^1(\Omega)} \leq 1$ such that
\begin{equation}\label{MO-embd}
\phj(x,\lam_0 r) \leq \theta(x,r) + h(x)
 \quad \mbox{ for a.e.\,} \ x \in \Omega \ \mbox{ and every } \ r \geq 0. 
\end{equation}
\end{lem}
In what follows, we shall consider $\alpha = \partial \phj: \Omega \times \R \to \R$ satisfying Assumption \ref{hp:alpha}. 
\begin{rem}[Examples of $\phj$]
A simple example may be 
$$
\phj(x,r) = \frac 1 {p(x)} |r|^{p(x)} \quad \mbox{ for } \ x \in \Omega \ \mbox{ and } \ r \in \R,
$$
which corresponds to the $L^{p(x)}(\Omega)$-space setting and falls within the framework developed in~\cite{AS}. As for more Musielak-Orlicz-type examples, we may consider
\begin{align*}
 \phj(x,r) = \frac 1 {p(x)} |r|^{p(x)} \left[ \log (1+|r|) \right]^{q(x)} \quad
  \mbox{ for } \ x \in \Omega \ \mbox{ and } \ r \in \R,
\end{align*}
with two measurable functions $p,q : \Omega \to \R$ satisfying 
\begin{equation*}
1 < p_- := \essinf_{x\in \Omega} p(x) \leq p_+ := \esssup_{x\in \Omega} p(x) < +\infty, \quad 
 1 \leq q_- \leq q_+ < +\infty.
\end{equation*}
Furthermore, noting that $\log(1+r) \leq \log(1+2r) \leq 2 \log(1+r)$ for $r \geq 0$, 
one can check that
$$
2^{p_-} \phj(x,r) \leq \phj(x,2r) \leq 2^{p_+ + q_+} \phj(x,r) \quad \mbox{ for } \ x \in \Omega
\ \mbox{ and } \ r \in \R.
$$
Since $p_+,q_+ < +\infty$ and $p_- > 1$, it follows from Proposition \ref{L:Del2-inv} that $\phj$ and $\phj^*$ 
fulfill the $\Delta_2$-condition. All the other conditions of Assumption \ref{hp:alpha} can be checked easily.
Similarly, one may also consider the multi-log nonlinearity, e.g., 
$$
\phj(x,r) = \frac 1 {p(x)} |r|^{p(x)} \left[ \log (1+|r|) \right]^{q(x)} [\log \log \e(1+|r|)]^{s(x)} 
\quad \mbox{ for } \ x \in \Omega \ \mbox{ and } \ r \in \R
$$
under natural assumptions on the variable exponents $p(\cdot)$, $q(\cdot)$, $s(\cdot)$, or 
the $(p(x),q(x))$-nonlinearity such as
$$
\phj(x,r) = \frac{\alpha(x)}{p(x)} |r|^{p(x)} + \frac{\beta(x)}{q(x)} |r|^{q(x)}
\quad \mbox{ for } \ x \in \Omega \ \mbox{ and } \ r \in \R
$$ 
for bounded measurable functions $\alpha, \beta : \Omega \to [0,+\infty)$ 
satisfying $\alpha + \beta \geq \delta$ a.e.~in $\Omega$ for some constant $\delta > 0$.
\end{rem}
We are now concerned with the following Cauchy-Dirichlet problem:
\begin{alignat}{4}
\alpha(x,u_t(x,t)) - \Delta_{m(x)} u(x,t) &= f(x,t) \quad &&\mbox{ for } \ x \in \Omega, \ t > 0,\label{pde1}\\
u(x,t) &= 0 \quad &&\mbox{ for } \ x \in \partial \Omega, \ t > 0,\label{bc1}\\
u(x,0) &= u_0(x) \quad &&\mbox{ for } \ x \in \Omega,\label{ic1}
\end{alignat}
where $\Delta_{m(x)}$ stands for the so-called \emph{$m(x)$-Laplace operator} given by
$$
\Delta_{m(x)} w(x) := \mathrm{div} \, \left(|\nabla w(x)|^{m(x)-2} \nabla w(x)\right)
$$
with a measurable function $m = m(x) : \Omega \to (1,+\infty)$ satisfying
\begin{equation}\label{p-bound}
1 < m^- :=  \essinf_{x \in \Omega}  m(x) \leq m^+ :=  \esssup_{x \in \Omega}  m(x) < +\infty
\end{equation}
and
\begin{equation}\label{log-hoelder}
 |m(x) - m(x')| \leq \frac{A}{\log(\e + 1/|x-x'|)}
  \quad \mbox{ for all } x,x' \in \Omega
\end{equation}
for some constant $A>0$. Here \eqref{log-hoelder} is called a \emph{logarithmic H\"older continuity} 
of the variable exponent $m(\cdot)$. Then, we set
$$
X = W^{1,m(x)}_0(\Omega), \quad \fhi(w) := \begin{cases}
					    \int_\Omega \frac 1 {m(x)} |\nabla w(x)|^{m(x)} \, \d x &
					    \mbox{ if } \ w \in W^{1,m(x)}_0(\Omega),\\
+\infty &\mbox{ otherwise},
					   \end{cases}
$$
where $W^{1,m(x)}_0(\Omega)$ is the closure of $C^\infty_c(\Omega)$ in $W^{1,m(x)}(\Omega)$ and 
$W^{1,m(x)}(\Omega)$ is a variable exponent Sobolev space (see~\cite{DHHR}) presented as a 
Musielak-Orlicz-Sobolev space with the choice of the generalized $\Phi$-function,
$$
\theta(x,r) := \frac 1 {m(x)} |r|^{m(x)}
$$
(see~\cite{DHHR} for more details), and equipped with the norm
$$
\|u\|_X := \||\nabla u|\|_{L^{m(x)}(\Omega)} \quad \mbox{ for } \ u \in X
$$
(see also~\cite[(a) of Theorem 8.2.4]{DHHR} for the Poincar\'e inequality). 
Then $\partial_\Omega \fhi(w)$ coincides with the extension onto 
$V = L^\phj(\Omega)$ of $-\Delta_{m(x)} w$ defined in the distributional sense. 
Thus \eqref{pde1}--\eqref{ic1} is reduced to the Cauchy problem for \eqref{eqn}. 
Moreover, $B = \partial_\Omega \fhi$ fulfills Assumption \ref{hp:B}, 
whenever $\varphi(x,r)$ defined above satisfies the assumption,
\begin{equation}\label{m(x)-hyp}
\phj(x,\lam_0 r) \leq \frac 1 {m^*(x)-\vep} |r|^{m^*(x)-\vep} + h(x)
\quad \mbox{ for } \ x \in \Omega \ \mbox{ and } \ r \geq 0,
\end{equation}
where $m^*(x) := Nm(x)/(N-m(x))_+$ is a counterpart of the \emph{Sobolev critical exponent}, 
for some $\lam_0 > 0$, $\vep > 0$ and $h \in L^1(\Omega)$ satisfying $\|h\|_{L^1(\Omega)} \leq 1$.
Indeed, under the logarithmic H\"older continuity of $m(\cdot)$, $W^{1,m(x)}_0(\Omega)$ is 
compactly embedded in $L^{m^*(x)-\vep}(\Omega)$ for $\vep > 0$, and hence, thanks to 
Lemma \ref{L:MO-embd} along with \eqref{m(x)-hyp}, we can deduce that 
$W^{1,m(x)}_0(\Omega)$ is compactly embedded in $V = L^\phj(\Omega)$. 
Hence (b) of Assumption \ref{hp:B} follows. 

\begin{rem}
The assumption \eqref{m(x)-hyp} may be relaxed by using the results in~\cite{KS,MOSS}.
\end{rem}
All the other conditions can be checked easily. Therefore our result reads,

\begin{thm}[Existence of strong solution to \eqref{pde1}--\eqref{ic1}]\label{T:ap1}
Suppose that the function $\alpha(x,t)$ satisfies {\rm Assumption \ref{hp:alpha}} and 
that $m(\cdot)$ fulfills \eqref{p-bound}, \eqref{log-hoelder} and \eqref{m(x)-hyp}. 
Then for any $f \in \calV^*$ and $u_0 \in W^{1,m(x)}_0(\Omega)$, the Cauchy-Dirichlet problem \eqref{pde1}--\eqref{ic1} 
admits at least one strong solution $u \in L^\phj(Q)$ in the sense of Theorem {\rm \ref{teo:main}}.
\end{thm}

Thanks to Theorem \ref{T:gen}, the existence result above can be extended to more general settings such as
\begin{alignat}{4}
\beta(x,u_t(x,t)) - \Delta_{m(x)} u(x,t) &= f(x,t) \quad &&\mbox{ for } x \in \Omega, \ t > 0.\label{pde1.5}
\end{alignat}
Here $\beta = \beta(x,r) : \Omega \times \R \to \R$ is a function 
which is measurable in $x$ for all $r \in \R$, maximal monotone in $r$ for a.e.~$x \in \Omega$,
and satisfies
\begin{alignat}{4}
\alpha_0 \phj(x,r) &\leq \beta(x,r) r + c_1(x) \quad &&\mbox{ for } \ r \in \R, \ x \in \Omega,\label{beta-1}\\
\phj^*(x,\beta(x,r)) &\leq C_2 \phj(x,r) + c_3(x) \quad &&\mbox{ for } \ r \in \R, \ x \in \Omega\label{beta-2}
\end{alignat}
for some $\alpha_0>0$, $C_2 \geq 0$ and $c_1,c_3 \in L^1(\Omega)$. Set $M : V \to V^*$ by
$$
z = M(u) \quad \stackrel{\text{define}}\Leftrightarrow \quad z(x) = \beta(x,u(x)) 
\ \mbox{ for a.e. } x \in \Omega
$$
for $u \in V$ and $z \in V^*$. Then $M$ turns out to be a maximal monotone operator 
satisfying (i) and (ii) of Assumption \ref{hp:M}. Hence we have

\begin{thm}[Existence of strong solution for \eqref{pde1.5}]\label{T:ap2}
In addition to the assumptions of Theorem {\rm \ref{T:ap1}}, let 
$\beta = \beta(x,r) : \Omega \times \R \to \R$ be a function measurable in $x$ 
for any $r \in \R$ and maximal monotone in $r$ for a.e.~$x \in \Omega$ such that
\eqref{beta-1} and \eqref{beta-2} hold true. Then for any $f \in L^{\phj^*}(Q)$
and $u_0 \in W^{1,m(x)}_0(\Omega)$, the Cauchy-Dirichlet problem for \eqref{pde1.5}
{\rm (}along with \eqref{bc1} and \eqref{ic1}{\rm )} admits at least one strong 
solution $u \in L^\phj(Q)$ in the sense of Theorem {\rm \ref{T:gen}.}
\end{thm}
\begin{rem}[Nonsmooth graphs]
The above result permits us to consider in particular the case when, for some, or all, $x\in\Omega$, 
$\beta(x,\cdot)$ is a nonsmooth maximal monotone graph (i.e., it contains vertical segments), which 
may occur, for instance, in some setting related to rate-independent problems or in some
class of phase-field models.
\end{rem}
Now, we shall discuss a generalization of the nonlinear elliptic operator.
To this end, we shall introduce the notion of Musielak-Orlicz-Sobolev spaces defined as follows:
\begin{df}[Musielak-Orlicz-Sobolev space]\label{D:MOS}
Let $\Omega$ be an open set in $\R^N$ and let $\theta$ be a generalized $\Phi$-function on $\Omega$. 
Then the \emph{Musielak-Orlicz-Sobolev space} $W^{1,\theta}(\Omega)$ is defined as
$$
W^{1,\theta}(\Omega) := \left\{ u \in L^\theta(\Omega) \colon \partial_{x_j} u \in L^\theta(\Omega) 
\ \mbox{ for } \ i = 1,2,\ldots,N\right\},
$$
where $\partial_{x_j} u$ stands for the distributional derivative of $u$, equipped with norm
$$
\|u\|_{W^{1,\theta}} := \|u\|_{L^\theta(\Omega)} + \||\nabla u|\|_{L^\theta(\Omega)}
\quad \mbox{ for } \ u \in W^{1,\theta}(\Omega).
$$
Furthermore, $W^{1,\theta}_0(\Omega)$ is defined as the closure of $C^\infty_c(\Omega)$ in $W^{1,\theta}(\Omega)$.
\end{df}
We next recall a compact embedding theorem for Musielak-Orlicz-Sobolev spaces, which is a counterpart of the 
well-known Rellich-Kondrachov theorem for standard Sobolev spaces. To this end, we introduce the notion of 
Matuszewska-Orlicz index of generalized $\Phi$-functions.
\begin{df}[Matuszewska-Orlicz index]
Under the setting of Definition {\rm \ref{D:MOS}}, set
\begin{equation}\label{Mat-1}
M(x,\lam) := \limsup_{r \to +\infty} \frac{\theta(x,\lam r)}{\theta(x,r)}
\quad \mbox{ for } \ x \in \Omega \ \mbox{ and } \ \lam > 1.
\end{equation}
Then the \emph{Matuszewska-Orlicz index} of $\theta$ is given by
\begin{equation}\label{Mat-2}
m(x) := \lim_{\lam \to +\infty} \frac{\log M(x,\lam)}{\log \lam} = \inf_{\lam > 1} \frac{\log M(x,\lam)}{\log \lam}
\quad \mbox{ for } \ x \in \Omega.
\end{equation}
Here the limit \eqref{Mat-1} is said to be \emph{uniform}, if for any $\vep > 0$ there exist $r_0>1$ and $\Lambda>1$ such that, for all $(x,\lam) \in \Omega\times [\Lambda,+\infty)$ and $r \geq r_0$, 
it holds that
$$
\left|M(x,\lam)-\frac{\theta(x,\lam r)}{\theta(x,r)}\right| < \vep.
$$
\end{df}
Furthermore, we recall a compact embedding theorem established in~\cite[Theorem 5.1]{LM19}:
\begin{thm}[Compact embedding of Musielak-Orlicz-Sobolev spaces]\label{T:cpt-emb}
Assume $\Omega$ be a bounded domain of $\R^N$ and let $\theta$ be a locally integrable 
generalized $\Phi$-function in $\Omega$. Suppose that the limits in \eqref{Mat-1} and \eqref{Mat-2}
are uniform with respect to $(x,\lambda)\in\Omega\times (\lam_0,\infty)$ for some $\lam_0 > 0$
and assume that the Matuszewska-Orlicz index $m(x)$ is the restriction onto $\Omega$ of a 
continuous function $\tilde{m}$ defined on the closure of $\Omega$. In addition, assume that
$$
1 < m_- := \inf_\Omega m
$$
and
\begin{equation}\label{theta}
\esssup_{x \in \Omega} \theta(x,r) \leq \Theta(r)
\end{equation}
for some function $\Theta:(0,+\infty)\to(0,+\infty)$. Then the embedding 
$W^{1,\theta}_0(\Omega) \hookrightarrow L^\theta(\Omega)$ is compact.

Furthermore, a variant of Poincar\'e's inequality holds true, i.e., there exists a constant $C$ 
depending only on $N$, $\Omega$, $\theta$ such that
\begin{equation*}
 \|u\|_{L^\theta(\Omega)} \leq C \||\nabla u|\|_{L^\theta(\Omega)} \quad 
 \mbox{ for } \ u \in W^{1,\theta}_0(\Omega).
\end{equation*}
Hence $\||\nabla \cdot|\|_{L^\theta(\Omega)}$ is an equivalent norm in $W^{1,\theta}_0(\Omega)$.
\end{thm}

\begin{rem}
As for standard Sobolev spaces, say $W^{1,m}_0(\Omega)$ for constant $1 < m < +\infty$, the conclusion 
of the theorem mentioned above corresponds to the compact embedding 
$W^{1,m}_0(\Omega) \hookrightarrow L^m(\Omega)$, which always holds true; however, the compact embedding 
$W^{1,\theta}_0(\Omega) \hookrightarrow L^\theta(\Omega)$ is not always true for Musielak-Orlicz-Sobolev 
spaces $W^{1,\theta}_0(\Omega)$ (see~\cite[Examples 3.1 and 3.2]{LM19} for counterexamples due to an 
$x$-dependence and oscillation of generalized $\Phi$-functions). Moreover, in~\cite{LM19}, the necessity 
of the assumptions, e.g., uniformity of the limits, \eqref{theta} and continuity of the index up to the 
boundary, is also discussed by giving counterexamples. 
\end{rem}

Now, we are ready to state a target equation,
\begin{alignat}{4}
\alpha(x,u_t(x,r)) - \mathrm{div} \,\mathbf{b}(x,\nabla u(x,t)) &= f(x,t) \quad &&\mbox{ for } \ x \in \Omega, \ t > 0,\label{pde2}\\
u(x,t) &= 0 \quad &&\mbox{ for } \ x \in \partial \Omega, \ t > 0,\label{bc2}\\
u(x,0) &= u_0(x) \quad &&\mbox{ for } \ x \in \Omega,\label{ic2}
\end{alignat}
where $\alpha : \Omega \times \R \to \R$ satisfies Assumption \ref{hp:alpha} and $\mathbf{b}:\Omega\times \R^N \to \R^N$ has a potential $\theta : \Omega\times \R \to \R$ which is even (i.e., $\theta(x,-r) = \theta(x,r)$), lower semicontinuous (indeed, of class $C^1$) and convex in the second variable such that
$$
\mathbf{b}(x,\xi) = \nabla_{\mathbf{\xi}} \theta(x,|\xi|) \quad \mbox{ for } \ x \in \Omega \ \mbox{ and } \ \xi \in \R^N.
$$ 
Moreover, assume that $\theta$ is a locally integrable generalized $\Phi$-function in $\Omega$ 
satisfying \eqref{MO-embd} and fulfilling
all the assumptions of Theorem \ref{T:cpt-emb}. Hence, from the facts we have reviewed so far, it holds that
\begin{equation*}
W^{1,\theta}_0(\Omega) \stackrel{\text{compact}}{\hookrightarrow} L^\theta(\Omega) \stackrel{\text{continuous}}\hookrightarrow L^\phj(\Omega).
\end{equation*}
Then we set 
$$
X=W^{1,\theta}_0(\Omega) \quad \mbox{ and } \quad
\fhi(w) := \begin{cases}
	    \int_\Omega \theta(x,|\nabla w(x)|) \, \d x &\mbox{ if } \ w \in W^{1,\theta}_0(\Omega),\\
	    +\infty &\mbox{ otherwise,}
	   \end{cases}
$$
which comply with Assumption \ref{hp:B} along with $\|\cdot\|_X := \||\nabla \cdot|\|_{L^\theta(\Omega)}$. Then the Cauchy-Dirichlet problem \eqref{pde2}--\eqref{ic2} is reduced to \eqref{eqn}, and therefore, thanks to Theorem \ref{teo:main}, we can assure existence of a strong solution to \eqref{pde2}--\eqref{ic2}.

It is worth mentioning that X.~Fan~\cite{Fan12NA} also provided a compact embedding theorem for 
Musielak-Orlicz-Sobolev spaces, where, under certain assumptions including smoothness for 
generalized $\Phi$-functions (not only in $r$ but also in $x$), it is proved that $W^{1,\theta}_0(\Omega)$ 
is compactly embedded in Musielak-Orlicz spaces $L^\psi(\Omega)$ for any generalized $\Phi$-functions 
$\psi$ satisfying $\psi \ll \theta^*$, i.e., for any $\lam > 0$, 
$\lim_{r\to+\infty} \psi(x,\lam r)/\theta^*(x,r) = +\infty$ uniformly for $x \in \Omega$. 
Here $\theta^*$ denotes the \emph{Sobolev conjugate function} of $\theta$ and corresponds 
to the Sobolev critical exponent for usual Sobolev spaces (with constant exponents). 
One can also apply the compact embedding theorem developed in~\cite{Fan12NA} and may 
obtain an existence result for \eqref{pde2}--\eqref{ic2} under a different frame of assumptions.

Finally, another possible application of the preceding results may be provided by the system
\begin{alignat*}{4}
\alpha(x,u_t(x,r)) + (-\Delta)^s u(x,t) &= f(x,t) \quad &&\mbox{ for } \ x \in \Omega, \ t > 0,\\
u(x,t) &= 0 \quad &&\mbox{ for } \ x \in \R^N \setminus \Omega, \ t > 0,\\
u(x,0) &= u_0(x) \quad &&\mbox{ for } \ x \in \Omega,
\end{alignat*}
where $0 < s < 1$ and $(-\Delta)^s$ is the so-called \emph{fractional Laplacian} defined through the following weak form:
$$
\left\langle (-\Delta)^s u, v \right\rangle_{\mathcal X^s} = \frac{C_s}2 \iint_{\R^N \times \R^N} \frac{(u(x)-u(y))(v(x)-v(y))}{|x-y|^{N+2s}} \, \d x \, \d y
$$
for $u,v \in \mathcal{X}^s := \{w \in H^s(\R^N) \colon w \equiv 0 \ \mbox{ in } \R^N \setminus \Omega \}$ equipped with $\|w\|_{\mathcal X^s} := (\|w\|_{L^2(\Omega)}^2 + \langle (-\Delta)^sw, w \rangle_{\mathcal X^s})^{1/2}$ for $w \in \mathcal{X}^s$ and where $C_s>0$ is a suitable constant (see~\cite[(3.2)]{Hitch}). Then set
$$
X = \mathcal{X}^s \quad \mbox{ and } \quad \fhi(w) = \frac{C_s}4 \iint_{\R^N \times \R^N} \frac{|w(x)-w(y)|^2}{|x-y|^{N+2s}} \, \d x \, \d y \quad \mbox{ for } \ w \in X.
$$
We refer the reader to, e.g.,~\cite{Hitch,SeVa13-1,SeVa13-2,ASS16,ASS19}, for more details on this kind of problems. Then Assumption \ref{hp:B} can also be checked for this setting. Indeed, it is clearly that $\fhi$ is proper, lower-semicontinuous, convex and coercive in $\mathcal X^s$ due to a Poincar\'e-type inequality (see, e.g.,~\cite[p.~9]{ASS16}). Furthermore, the space $X$ turns out to be compactly embedded in $V$, provided that \eqref{m(x)-hyp} with $m^*(x)$ replaced by $2N/(N-2s)_+$ is satisfied (indeed, $H^s(\Omega)$ is compactly embedded in $L^q(\Omega)$ for any $q \in [1,2N/(N-2s)_+)$, see also Lemma \ref{L:MO-embd}).

\section*{Acknowledgments}
G.~Akagi has been supported by JSPS KAKENHI Grants Number JP21KK0044, JP21K18581, JP20H01812, JP18K18715, JP16H03946, JP20H00117, JP17H01095, by the Alexander von Humboldt Foundation, and by the Carl Friedrich von Siemens Foundation. He is also deeply grateful to the Helmholtz Zentrum M\"unchen and the Technische Universit\"at M\"unchen for their kind hospitality and support during his stay in Munich. G.~Schimperna has been partially supported by GNAMPA (Gruppo Nazionale per l'Analisi Matematica, la Probabilit\`a e le loro Applicazioni) of INdAM (Istituto Nazionale di Alta Matematica). This work was supported by the Research Institute for Mathematical Sciences, an International Joint Usage/Research Center located in Kyoto University.

\end{document}